\documentclass{amsart}
\usepackage[T1]{fontenc}
\usepackage{enumerate}
 \usepackage{amsthm,amsmath,amssymb}
\usepackage[all,cmtip]{xy}
\usepackage{mathrsfs}
\usepackage{amsfonts}

 \usepackage[mathcal]{euscript}

\def \qed { $\hfill \square$ \vskip 8 pt}
\newcommand{\RN}{\mathbb{R}^N}
\newcommand{\R}{\mathbb{R}}

\newcommand{\longto}{\longrightarrow}
\newcommand{\G}{\mathbb{G}}
\newcommand{\g}{\mathfrak{g}}
\newcommand{\e}{\varepsilon }

\newcommand{\LL}{\mathcal{L}}
\renewcommand{\d}{\mathrm{d}}
\newcommand{\de}{\partial}

\newcommand{\lan}{\langle}
\newcommand{\ran}{\rangle}
\renewcommand{\H}{\mathcal{H}}

\allowdisplaybreaks[1]
\newtheorem{theorem}{Theorem}[section]
\newtheorem{lemma}[theorem]{Lemma}

\newtheorem{corollary}[theorem]{Corollary}

\theoremstyle{remark}
\newtheorem{remark}[theorem]{Remark}

\theoremstyle{remark}
\newtheorem{example}[theorem]{Example}

\theoremstyle{theorem}

\theoremstyle{theorem}

\theoremstyle{definition}
\newtheorem{definition}[theorem]{Definition}

\numberwithin{equation}{section}
\begin{document}
\title[Harnack inequality for subelliptic operators]
 {The Strong Maximum Principle and \\ the Harnack inequality for a class of\\
 hypoelliptic divergence-form operators}
 \author{Erika Battaglia}
 \address{Dipartimento di Matematica,
         Universit\`{a} degli Studi di Bologna\\
         Piazza di Porta San Donato, 5 - 40126 Bologna, Italy\\
         Fax.: +39-51-2094490.}
          \email{erika.battaglia3@unibo.it}
 \author{Stefano Biagi}
 \address{Dipartimento di Matematica,
         Universit\`{a} degli Studi di Bologna\\
         Piazza di Porta San Donato, 5 - 40126 Bologna, Italy\\
         Fax.: +39-51-2094490.}
 \email{stefano.biagi3@unibo.it}
 \author{Andrea Bonfiglioli}
 \address{Dipartimento di Matematica,
         Universit\`{a} degli Studi di Bologna\\
         Piazza di Porta San Donato, 5 - 40126 Bologna, Italy\\
         Tel.: +39-51-2094498, Fax.: +39-51-2094490.}
 \email{andrea.bonfiglioli6@unibo.it}

\begin{abstract}
 In this paper we consider a class
 of hypoelliptic second-order partial differential operators $\LL$
 in divergence form on $\RN$, arising from CR geometry and Lie group theory, and we prove
 the Strong and Weak Maximum Principles and the Harnack Inequality for $\LL$.
 The involved operators are not assumed to belong to the H\"ormander hypoellipticity class,
 nor to satisfy subelliptic estimates, nor Muckenhoupt-type estimates on the degeneracy of the
 second order part; indeed our results
 hold true in the infinitely-degenerate case and for operators which are not
 necessarily sums of squares.
 We use a Control Theory result on hypoellipticity in order to
 recover a meaningful geometric information on connectivity and maxima propagation, yet
 in the absence of any H\"ormander condition.
 For operators $\LL$ with $C^\omega$ coefficients, this control-theoretic result
 will also imply a Unique Continuation property for the $\LL$-harmonic functions.
 The (Strong) Harnack Inequality is obtained via the Weak Harnack Inequality
 by means of a Potential Theory argument, and by a crucial use of the Strong Maximum Principle
 and the solvability of the Dirichlet problem for $\LL$ on a basis of the
 Euclidean topology.
\end{abstract}

 \keywords{Degenerate-elliptic operators; Maximum principles; Harnack inequality; Unique continuation; Divergence form operators.}
 \subjclass[2010]{Primary: 35B50, 35B45, 35H20; Secondary: 35J25, 35J70, 35R03}


\maketitle

\section{Introduction and main results}\label{sec:introductionAB}
 Throughout the paper, we shall be concerned with linear second order
 partial differential operators (PDOs, in the sequel), possibly degenerate-elliptic, of the form
\begin{equation}\label{mainLL}
    \LL:=\frac{1}{V(x)}\sum_{i,j=1}^N\frac{\de}{\partial {x_i}}\Big(V(x)\,a_{i,j}(x)\,\frac{\de}{\partial {x_j}}\Big),\qquad
    x\in\RN,
\end{equation}
 where $V$ is a $C^\infty$
 positive function on $\RN$, the matrix $A(x):=(a_{i,j}(x))_{i,j}$
 is symmetric and \emph{positive semi-definite} at every point
 $x\in\RN$, and it has real-valued $C^\infty$ entries. In particular,  $\LL$ is formally self-adjoint
 on $L^2(\RN,\d\nu)$ with respect to the measure $\d\nu(x)=V(x)\,\d x$, which clarifies the r\^ole of $V$.
 We tacitly understand these structural assumptions on $\LL$ throughout.
 The literature on divergence-form operators like \eqref{mainLL} in
 the strictly-elliptic case is so vast that we do no attempt to collect the related references.
 Instead, we mention some papers (relevant for the topics of the present paper) in the degenerate case.

 Degenerate-elliptic operators of the form \eqref{mainLL} were
 extensively studied by Jerison and S\'anchez-Calle in the paper \cite{JerisonSanchez-Calle}
 (under a suitable subelliptic assumption), where it is also described how these PDOs naturally intervene in the study of function theory
 of several complex variables and CR Geometry (see also \cite{FollandStein, Kohn65, RothschildStein}).
 Prototypes for the PDOs
 \eqref{mainLL} also arise in the theory of sub-Laplace operators on real Lie groups
 (e.g., for Carnot groups, \cite{BLUlibro}), as well as in Riemannian Geometry (e.g., the
 Laplace-Beltrami operator has the form ${\sqrt {|g|^{-1}}}\sum
 \partial_i (\sqrt{|g|} g^{ij} \partial_j )$).
 Re\-gu\-la\-ri\-ty issues for degenerate-elliptic divergence-form operators
 comprising the Harnack Ine\-qua\-li\-ty and the Maximum Principles (to which this paper is devoted)
 trace back to the 80's, with the deep investigations by:
 Fabes, Kenig, Serapioni \cite{FabesKenigSerapioni};
 Fabes, Jerison, Kenig \cite{FabesJerisonKenig, FabesKenigJerison};
 Gutiérrez \cite{Gutierrez}.
 In these papers, operators as in \eqref{mainLL} are considered (with $V\equiv 1$) with low regularity assumptions on the coefficients,
 under the hypothesis that the degeneracy of $A(x)$ be controlled on both sides by some Muckenhoupt weight.

 Recent investigations on the Harnack inequality for variational operators, comprising \eqref{mainLL} as a special case,
 also assume Muckenhoupt weights on the degeneracy; see \cite{DeCiccoVivaldi, Zamboni}.
 Very recently, a systematic study of the Potential Theory for the harmonic/subharmonic functions related to operators $\LL$
 as in \eqref{mainLL}
 has been carried out in the series of papers \cite{AbbondanzaBonfiglioli, BattagliaBonfiglioli, BonfLancJEMS, BonfLancTommas},
 under the assumption that $\LL$ possesses a (smooth) global positive fundamental solution.\medskip

 We remark that in the present paper we do not require $\LL$ to be a H\"ormander operator,
 our results holding true in the infinitely-degenerate case as well, nor we make
 any assumption of subellipticity or Muckenhoupt-weighted degeneracy
 (see Example \ref{exa.Lie2}); furthermore, we do not assume the existence of a global
 fundamental solution for $\LL$. Hence our results are not contained in any of the
 aforementioned papers.\medskip

 We now describe the main results of this paper concerning $\LL$, namely the \emph{Strong Maximum Principle}
 and the \emph{Harnack Inequality} for $\LL$; gradually as we need to specify them,
 we introduce the three assumptions under which our theorems are proven.
 As we shall see in a moment, the main hypothesis is a hypoellipticity assumption.\medskip

 In obtaining our main results we are much indebted to the ideas in the pioneering paper
 by Bony, \cite{Bony}, where H\"ormander operators are considered. The main novelty of our
 framework is that we have to renounce to the geometric information encoded in H\"ormander's Rank Condition:
 the latter
 implies a connectivity/propagation property (leading to the Strong Maximum Principle),
 as well as it implies hypoellipticity, due to the well-known H\"ormander's theorem \cite{Hormander}.
 In our setting, the approach is somewhat reversed: hypoellipticity is the main assumption,
 and we need to derive from it some appropriate connectivity and propagation features, even in the absence
 of a maximal rank condition. This will be made possible by exploiting a Control Theory result by Amano \cite{Amano} on hypoelliptic PDOs, as we shall
 describe in detail. Once the Strong Maximum Principle is established, the path to the (Strong) Harnack Inequality
 is traced in \cite{Bony}: we pass through the solvability of the Dirichlet problem, the relevant Green kernel and a
 Weak Harnack Inequality. Finally, the gap between the Weak and Strong Harnack Inequalities is filled by
 an abstract Potential Theory result, due to Mokobodzki and Brelot, \cite{Brelot}.\medskip

 In order to describe our results more closely, we first fix some notation and definition: we say that a linear second order PDO on $\RN$
\begin{equation}\label{Lgenerale}
 L:=\sum_{i,j=1}^N\alpha_{i,j}(x)\frac{\de^2}{\de x_i\de x_j}+
  \sum_{i=1}^N \beta_i(x)\frac{\de}{\de x_i}+\gamma(x)
\end{equation}
  is \emph{non-totally degenerate} at a point $x\in\RN$ if the matrix $(\alpha_{i,j}(x))_{i,j}$ (which will be referred to
 as the principal matrix of $L$)
 is non-vanishing. We observe that the principal matrix of an operator
 $\LL$ of the form \eqref{mainLL} is precisely $A(x)=(a_{i,j}(x))_{i,j}$.
 We also recall that $L$ is said to be
 ($C^\infty$-)hypoelliptic in an open set $\Omega\subseteq\RN$ if, for every $u\in \mathcal{D}'(\Omega)$,
 every open set $U\subseteq \Omega$ and every $f\in C^\infty(U,\R)$, the equation $L u =f$ in $U$
 implies that $u$ is (a function-type distribution associated with) a
 $C^\infty$ function on $U$.

 In the sequel, if $\Omega\subseteq\RN$ is open, we say that $u$
 is \emph{$L$-harmonic} (resp., \emph{$L$-subharmonic})
 in $\Omega$ if $u\in C^2(\Omega,\R)$ and $L u=0$ (resp., $L u\geq 0$) in $\Omega$.
 The set of the $L$-harmonic functions in $\Omega$ will be denoted by $\H_{L}(\Omega)$.
 We observe that, if $L$ is hypoelliptic on every open subset of $\RN$, then
 $\H_{L}(\Omega)\subset C^\infty(\Omega,\R)$; under this hypoellipticity  assumption,
 $\H_{L}(\Omega)$ has important topological properties, which will be crucially used in the sequel (Remark \ref{rem.topolocinfyt}).\medskip

 In order to introduce our first main result we assume the following hypotheses on $\LL$:
 \begin{description}
   \item[(NTD)]
   $\LL$ is \emph{non-totally degenerate at every point of $\RN$}, or equivalently (recalling that $A(x)$ is symmetric
 and positive semi-definite),
\begin{equation}\label{NTD}
 \textrm{trace}(A(x))> 0,\quad \text{for every $x\in \RN$.}
\end{equation}

   \item[(HY)]
   $\LL$ is $C^\infty$-\emph{hypoelliptic in every open subset of $\RN$}.
 \end{description}
 Under these two assumptions we shall prove the \emph{Strong Maximum Principle for $\LL$}.\medskip

 Condition (NTD), if compared with the above mentioned Muckenhoupt-type weights on the degeneracies of $A(x)$,
 does not allow a \emph{simultaneous} vanishing of the eigenvalues of $A(x)$, but it has the advantage
 of permitting a very fast vanishing of the smallest eigenvalue (see Example
 \ref{exa.Lie2}) together with a very fast growing of the largest one (see Example
 \ref{exa.Lie}); both phenomena can happen at an exponential rate (e.g.,
 like $e^{-1/x^2}$ as $x\to 0$ in the first case, and like
 $e^{x}$ as $x\to \infty$ in the second case), which is not allowed when Muckenhoupt weights are involved.

 Meaningful examples of operators satisfying hypotheses (NTD) and (HY), providing prototype PDOs to which our theory applies and a motivation for our
 investigation, are now described
 in the following two examples.
\begin{example}\label{exa.Lie}
 The following PDOs satisfy the assumptions (NTD) and (HY).\medskip

  (a.)\,\, If $\RN$ is equipped with a Lie group structure $\G=(\RN,*)$,
  and if we fix a set $X:=\{X_1,\ldots,X_m\}$ of Lie-generators for the Lie algebra $\g$ of $\G$
  (this means that the smallest Lie algebra containing $X$ is equal to $\g$), then a direct computation shows that
 \begin{equation}\label{subla}
   \LL_X:=-\sum_{j=1}^m X_j^*\, X_j
 \end{equation}
  is of the form \eqref{mainLL}, where $V(x)$ is the density of the Haar measure $\nu$ on $\G$,
  and $(a_{i,j})_{i,j}$ is equal to $S\,S^T$, where $S$ is the $N\times m$ matrix whose columns
  are given by the coefficients of the vector fields $X_1,\ldots,X_m$;
   here $X_j^*$ denotes the (formal) adjoint of $X_j$ in the Hilbert space $L^2(\R^N,\d\nu)$.
  Most importantly, $\LL_X$ in \eqref{subla} satisfies the assumptions (NTD) and (HY) above. Indeed:
  \begin{itemize}
    \item The non-total-degeneracy is a consequence of $X$ being a set of Lie-generators of $\g$.

    \item
   $\LL_X$ is a H\"ormander operator, of the form
    $\sum_{j=1}^m X_j^2+X_0$,
    where $X_0$ is a linear combination (with smooth coefficients) of $X_1,\ldots,X_m$.
    Therefore $\LL_X$ is hypoelliptic due to H\"ormander's Hypoellipticity Theorem, \cite{Hormander},
   jointly with the cited fact that $X$ is a set of Lie-generators of $\g$.
  \end{itemize}
  The density $V$ need not be identically $1$ as for example for the Lie group $(\R^2,*)$, where
  $$(x_1,x_2)*(y_1,y_2)=(x_1+y_1e^{x_2},x_2+y_2), $$
  since in this case $V(x)=e^{-x_2}$. The left-invariant PDO associated with the set of generators
  $X=\{e^{x_2}\frac{\de}{\de x_1},\frac{\de}{\de x_2}\}$ has fast-growing coefficients:
  $$\LL_X=e^{2x_2}\frac{\de^2}{\de x_1^2}+\frac{\de^2}{\de x_2^2}-\frac{\de}{\de x_2}. $$
 Note that the eigenvalues of the principal matrix of $\LL_X$ are $e^{2x_2}$ and $1$, so that the largest eigenvalue
 cannot be controlled (for $x_2>0$) by any integrable weight.\medskip

  (b.)\,\, More generally (arguing as above), if $X=\{X_1,\ldots,X_m\}$ is a family of smooth vector fields in $\RN$
  satisfying H\"ormander's Rank Condition, if $\d\nu(x)=V(x)\,\d x$ is the Radon measure associated
  with any positive smooth density $V$ on $\RN$, then the operator
  $-\sum_{j=1}^m X_j^*\, X_j$
    is of the form \eqref{mainLL} and it satisfies (NTD) and (HY).
    Here $X_j^*$ denotes the formal adjoint of $X_j$ in $L^2(\R^N,\d\nu)$.
    As already observed, PDOs of this form naturally arise in CR Geometry and in the function theory of
    several complex variables (see \cite{JerisonSanchez-Calle}).
\end{example}
 The above examples show that geometrically meaningful
 PDOs belonging to the class of our concern actually fall in the hypoellipticity class
 of the H\"ormander operators. Nonetheless, hypotheses (NTD) and (HY) are
 general enough to comprise \emph{non-H\"ormander} and \emph{non-subelliptic} PDOs, as it is shown in the next example.
 Applications to this kind of \emph{infinitely-degenerate} PDOs
 also furnish one of the main motivation for our study.
\begin{example}\label{exa.Lie2}
 Let us consider the class of operators in $\R^2$ defined by
\begin{subequations}
\begin{equation}\label{fediiOP}
 \LL_a=\frac{\de^2}{\de x_1^2}+\Big(a(x_1)\,\frac{\de}{\de x_2}\Big)^2,
\end{equation}
 with $a\in C^\infty(\R,\R)$, $a$ even, nonnegative, nondecreasing on $[0,\infty)$ and vanishing only at $0$. Then
 $\LL_a$ satisfies (NTD) (obviously) and (HY), thanks to a result
 by Fedi$\breve{\textrm{\i}}$, \cite{Fedii}. Note that $\LL_a$ does not satisfy H\"ormander's Rank Condition
 at $x_1=0$ if all the derivatives of $a$ vanish at $0$, as for
 $a(x_1)=\exp(-1/x_1^2)$. Other examples of operators satisfying our assumptions
 (NTD) and (HY) but failing to be H\"ormander operators
  can be found, e.g., in the following papers: Bell and Mohammed \cite{BellMohammed}; Christ \cite[Section 1]{Christ};
  Kohn \cite{Kohn}; Kusuoka and Stroock \cite[Theorem 8.41]{KusuokaStroock};
  Morimoto \cite{Morimoto}. Explicit examples are, for instance,
  \begin{align}
   &\frac{\de^2}{\de x_1^2}+\Big(\exp(-1/|x_1|)\,\frac{\de}{\de x_2}\Big)^2+ \Big(\exp(-1/|x_1|)\,\frac{\de}{\de x_3}\Big)^2&&\quad \text{in $\R^3$,}\label{christOP}\\
   &\frac{\de^2}{\de x_1^2}+\Big(\exp(-1/\sqrt{|x_1|})\,\frac{\de}{\de x_2}\Big)^2+ \frac{\de^2}{\de x_3^2}&&\quad \text{in $\R^3$,}\label{kusuokastroockOP}\\
   &\frac{\de^2}{\de x_2^2}+\Big(x_2\,\frac{\de}{\de x_1}\Big)^2+
   \frac{\de^2}{\de x_4^2}+\Big(\exp(-1/\sqrt[3]{|x_1|})\,\frac{\de}{\de x_3}\Big)^2&&\quad \text{in $\R^4$}\label{morimotoOP}.
 \end{align}
\end{subequations}
 For the hypoellipticity of \eqref{christOP} see \cite{Christ}; for \eqref{kusuokastroockOP} see \cite{KusuokaStroock}; for \eqref{morimotoOP} see \cite{Morimoto}.
 Later on, in proving the Harnack Inequality, we shall add another hypothesis to (NTD) and (HY) and, as we shall show,
 the operators from \eqref{fediiOP} to \eqref{morimotoOP} (and those in Example \ref{exa.Lie}) will fulfil this assumption as well. Hence the main results of this paper
 (except for the Unique Continuation result in Section \ref{sec:Harnack_analytic}, proved for operators with $C^\omega$ coefficients) fully apply to
 these PDOs.

 Moreover, since the PDOs \eqref{fediiOP}-to-\eqref{morimotoOP} \emph{are not subelliptic} (see Remark \ref{rem.equivHYeps}),
 they do not fall in the class considered by Jerison and S\'anchez-Calle in \cite{JerisonSanchez-Calle}.
 Finally, note that the smallest eigenvalue in all the above examples vanishes
 very quickly (like $\exp(-1/|x|^\alpha)$ for $x\to 0$, with positive $\alpha$) and it cannot be bounded
 from below by any weight $w(x)$ with locally integrable reciprocal function.
 \end{example}
 Our first main result under conditions (NTD) and (HY) is the following one.
\begin{theorem}[\textbf{Strong Maximum Principle for $\LL$}]\label{th:SMP}
 Suppose that $\LL$ is an operator of the form
 \eqref{mainLL}, with $C^\infty$ coefficients $V>0$ and $(a_{i,j})_{i,j}\geq 0$,
 and that it satisfies \emph{(NTD)} and \emph{(HY)}. Let
 $\Omega\subseteq\RN$ be a connected open set. Then,
 the following facts hold.
\begin{enumerate}
  \item[\emph{(1)}]  Any function
 $u\in C^2(\Omega,\R)$ satisfying $\LL u\geq 0$ on $\Omega$
 and attaining a maximum in $\Omega$ is constant
 throughout $\Omega$.\medskip

  \item[\emph{(2)}]
  If $c\in C^\infty(\RN,\R)$ is nonnegative on $\RN$, and if we set
 \begin{equation}\label{LC}
    \LL_c:=\LL-c,
 \end{equation}
 then any function
 $u\in C^2(\Omega,\R)$ satisfying $\LL_c u\geq 0$ on $\Omega$ and
 attaining a nonnegative maximum in $\Omega$ is constant
 throughout $\Omega$.
\end{enumerate}
\end{theorem}
 \noindent The r\^ole of the nonnegativity of the zero-order term $c$ in the above statement (2)
 in obtaining Strong Maximum Principles is well-known (see e.g., Pucci and Serrin \cite{PucciSerrin}).
\begin{remark}\label{rem.WMPvale}
 (a.)\,\, Obviously, the Strong Maximum Principle (SMP, shortly) in Theorem \ref{th:SMP} will immediately provide
 the \emph{Weak} Maximum Principle (WMP, shortly) for operators $\LL$ and $\LL-c$, for any nonnegative zero-order term $c$ (and any bounded
 open set $\Omega$), see
 Corollary \ref{th.WMPPP} for the precise statement.\medskip

 (b.)\,\, We will show that, in order to obtain the SMP and WMP for $\LL-c$,
 it is also sufficient to replace the hypothesis on the hypoellipticity of $\LL$ with the (more natural hypothesis of the) hypoellipticity of $\LL-c$,
 still under assumption (NTD) and the divergence-form structure of $\LL$; see Remark \ref{PMFancheconc2} for the precise result.
\end{remark}
 Our proof of the SMP in Theorem \ref{th:SMP} follows a rather classical scheme, in that it rests
 on a Hopf Lemma for $\LL$ (see Lemma \ref{lem_Hopf}). However, the passage
 from the Hopf Lemma to the SMP is, in general, non-trivial and the same is true in our framework.
 For example, in the paper \cite{Bony} by Bony, where
 H\"ormander operators are considered, this passage is accomplished by means of a maximum propagation
 principle, crucially based on H\"ormander's Rank Condition, the latter ensuring a connectivity property
 (the so-called \emph{Chow's Connectivity Theorem} for H\"ormander vector fields).
 The novelty in our setting is that, since hypotheses (NTD) and (HY)
 do \emph{not} necessarily imply that $\LL$ is a H\"ormander operator (see for instance Example
 \ref{exa.Lie2}), we have to supply for a lack of geometric information.
 Due to this main novelty, we describe more closely our argument in deriving the SMP.

 As anticipated, we are able to supply the lack of H\"ormander's Rank Condition by using
 a notable control-theoretic property (seemingly long-forgotten in the PDE literature), encoded in the hypoellipticity assumption (HY), proved by Amano in
 \cite{Amano}: indeed,
  thanks to the hypothesis (NTD), we are entitled to use
  \cite[Theorem 2]{Amano} which states that
  (HY) ensures the \emph{controllability} of the ODE system
\begin{equation*}
 \dot \gamma= \xi_0 X_0(\gamma)+\sum_{i=1}^N \xi_i X_i(\gamma),\qquad (\xi_0,\xi_1,\ldots,\xi_N)\in \R^{1+N},
\end{equation*}
 on every open and connected subset of $\RN$. Here $X_1,\ldots,X_N$
 denote the vector fields associated with the rows of the principal matrix of $\LL$, whereas $X_0$
 is the drift vector field obtained by writing $\LL$ (this being always possible) in the form
 $$\LL u= \sum_{i=1}^N  \frac{\de}{\de x_i}( X_iu)+ X_0u.$$
 By definition of a controllable system, Amano's controllability result provides another
 geometric \emph{con\-nec\-ti\-vi\-ty property} (a substitute for Chow's Theorem): any couple of points
 can be joined by a continuous path which is piece-wise an
 integral curve of some vector field $Y$ be\-long\-ing to $\textrm{span}_\R\{X_0,X_1,\ldots,X_N\}$.
 The SMP will then follow if we show that there is a pro\-pa\-ga\-tion
 of the maximum of any $\LL$-subharmonic function $u$ along all integral curves $\gamma_Y$ of
 every  $Y\in \textrm{span}_\R\{X_0,X_1,\ldots,X_N\}$.
 In other words, we need to show that if the set $F(u)$
 of the maximum points of $u$ intersects any such $\gamma_Y$, then $\gamma_Y$
 is wholly contained in $F(u)$: briefly, if this happens we say that $F(u)$
 is $Y$-invariant. In its turn, this $Y$-invariance property
 can be characterized (see Bony, \cite[\S 2]{Bony})
 in terms of a tangentiality property of $Y$ with respect to $F(u)$
 (the  reader is referred to Section \ref{sec:SMP} below for this notion of
 tangentiality).

 Now, the self-adjoint structure of our PDO $\LL$ in \eqref{mainLL}
 ensures that $X_0$ is a linear combination with smooth coefficients of
 $X_1,\ldots,X_N$. Hence, by the very definition of tangentiality
 (see e.g., \eqref{daprovSMP1primaa}),
 the tangentiality of $X_0$ w.r.t.\,$F(u)$
 will be inherited from the tangentiality of $X_1,\ldots,X_N$ w.r.t.\,$F(u)$.
 By means of the above argument of controllability/propagation, this allows us
 to reduce the proof of the SMP to showing that
 any of the vector fields $X_1,\ldots,X_N$ is tangent to $F(u)$.
 Luckily, this tangentiality is a consequence of the choice of $X_1,\ldots,X_N$ as deriving from the rows
 of the principal matrix of $\LL$, together with the Hopf-type Lemma \ref{lem_Hopf}
 for $\LL$. This argument is provided, in all detail, in Section \ref{sec:SMP}.\medskip

 The use of the above ideas, plus the classical Holmgren's Theorem, will allow us to prove that,
 when $\LL$ has real-analytic coefficients, a \emph{Unique Continuation} result holds true for $\LL$:
 any $\LL$-harmonic function defined on a connected open set $U$ which
 vanishes on some non-void open subset is necessarily null on the whole of $U$ (see Theorem
 \ref{th:Uniquecont}).   We observe that the $C^\omega$ assumption is satisfied, for example,
  if $\LL$ is a left invariant operator on a Lie group (e.g., a sub-Laplacian on a Carnot group, as in
  \cite{BLUlibro}), since, as it is well-know, any Lie group can be endowed with a compatible $C^\omega$ structure.\medskip
\begin{remark}
 We explicitly remark that, as it is proved by Amano in \cite[Theorem 1]{Amano}, the above
 controllability property ensures the validity of the H\"ormander Rank Condition only
 on an open \emph{dense}  subset of $\RN$
 which may fail to coincide with the whole of $\RN$. This actual possible
 lack of the H\"ormander Rank Condition is clearly exhibited in
 Example \ref{exa.Lie2} (of non-H\"ormander operators which nonetheless satisfy our assumptions
 (NTD) and (HY), and hence the SMP).

 To the best of our knowledge, Amano's controllability result for hypoelliptic
 non-totally-degenerate operators has been long forgotten in the literature; only recently,
 it has been used by the third-named author and B.\,Abbondanza \cite{AbbondanzaBonfiglioli}
 in studying the Dirichlet problem for $\LL$,
 and in obtaining Potential Theoretic
 results for the harmonic sheaf related to $\LL$.
\end{remark}

 In order to give the second main result of the paper (namely, the \emph{Harnack
 Inequality} for $\LL$), we shall need a further assumption, very similar to (HY)
 (and, indeed, equivalent to it in many important cases), together with some technical results
 on the solvability of the Dirichlet problem related to $\LL$. Our next assumption is the following one:\medskip
 \begin{description}
   \item[$\textrm{(HY)}_\e$]
   \emph{There exists $\e>0$ such that
   $\LL-\e$ is $C^\infty$-hypoelliptic in every open subset of $\RN$}.\medskip
 \end{description}
 For operators $\LL$ satisfying hypotheses (NTD), (HY) and (HY)$_\e$ we are able to prove the
 Harnack Inequality (see Theorem \ref{lem.crustimabassoTHEOFORTE}).\medskip

 We postpone the description of the relationship between assumptions (HY) and $\textrm{(HY)}_\e$
 (and their actual equivalence for large classes of operators: for subelliptic PDOs, for instance)
 in Remark \ref{rem.equivHYeps} below.
 Instead, we anticipate the r\^ole of the perturbation $\LL-\e$
 of the operator $\LL$: this is motivated by a crucial comparison argument
 (which we generalize to our setting), due to Bony \cite[Proposition 7.1, p.298]{Bony},
 giving the lower bound
\begin{equation}\label{lowboundkeps}
 u(x_0)\geq \e \int_{\Omega} u(y)\,k_\e(x_0,y)\,V(y)\,\d y\qquad \forall\,x_0\in\Omega,
\end{equation}
 for every nonnegative $\LL$-harmonic function $u$ on the open set $\Omega$ which possesses
 a Green kernel $k_{\e}(x,y)$ relative to the perturbed operator $\LL-\e$
 (see Theorem \ref{th.greeniani} for the notion of a Green kernel, and see
 Lemma \ref{lem.crustimabasso} for the proof of \eqref{lowboundkeps}).
 This lower bound, plus some topological facts on hypoellipticity, is the key ingredient for a \emph{Weak} Harnack Inequality
 related to $\LL$, as we shall explain shortly.

 Some remarks on assumption (HY)$_\e$ are now in order.
\begin{remark}\label{rem.equivHYeps}
 Hypothesis (HY)$_\e$ is implicit in hypothesis (HY) for notable classes of operators, whence
 our assumptions for the validity of the Harnack Inequality for $\LL$ reduce to (NTD) and (HY) solely: namely,
 \emph{(HY) implies (HY)$_\e$ in the following cases:}\medskip

 \begin{itemize}
 \item for H\"ormander operators, and, more generally, for second order \emph{subelliptic} operators
 (in the usual sense of fulfilling a subelliptic estimate,
 see e.g., \cite{JerisonSanchez-Calle, Kohn});
 indeed, any operator $L$ in these
 classes of PDOs is hypoelliptic (see H\"ormander \cite{Hormander}, Kohn and Nirenberg \cite{KohnNirenberg}),
 and $L$ still belongs to these classes after the addition of a smooth zero-order term;\medskip

 \item for operators with \emph{real-analytic coefficients.}
 Indeed, in the $C^\omega$ case, one can apply known results by Ole\u{\i}nik and Radkevi\v{c} ensuring that,
 for a general $C^\omega$ operator $L$ as in \eqref{Lgenerale}, hypoellipticity is equivalent
 to the verification of H\"ormander's Rank Condition for the vector fields $X_0,X_1,\ldots,X_N$ obtained by
 rewriting $L$ as $\sum_{i=1}^N \de_i(X_i)+X_0+\gamma$; this
 condition is clearly invariant under any change of the zero-order term $\gamma$ of $L$ so that
 (HY) and (HY)$_\e$ are indeed equivalent.\medskip

 \end{itemize}
 The problem of establishing, in general, whether (HY) implies (HY)$_\e$ seems non-trivial
 and it is postponed to future investigations.\footnote{It appears that having some quantitative information on the loss of derivatives
 may help in facing this question (personal communication by A. Parmeggiani).} In this regard we recall that, for example,
 in the complex coefficient case the presence of a zero-order term (even a small $\varepsilon$)
 may drastically alter hypoellipticity (see for instance the example given by Stein in \cite{Stein}).

 We explicitly remark that the operators
 \eqref{fediiOP}-to-\eqref{morimotoOP} are \emph{not} subelliptic (nor $C^\omega$), yet they satisfy
 hypotheses (NTD), (HY) and (HY)$_\e$. The lack of subellipticity  is a consequence of the
 characterization of the subelliptic PDOs due to Fefferman and Phong \cite{FeffermanPhong2, FeffermanPhong}
 (see also \cite[Prop.1.3]{Kohn} or \cite[Th.2.1 and Prop.2.1]{JerisonSanchez-Calle},
 jointly with the presence of a coefficient with a zero of infinite order
 in \eqref{fediiOP}-to-\eqref{morimotoOP}). The second assertion concerning the verification
 of (HY)$_\e$ (the other hypotheses being already discussed) derives from the following result by Kohn, \cite{Kohn}:
 any operator of the form
 $$L_1+\lambda(x)\,L_2\quad \text{in $\R^n_x\times \R^m_y$}$$
 is hypoelliptic, where $\lambda\in C^\infty(\R_x)$, $\lambda\geq 0$ has
 a zero of infinite order at $0$
 (and no other zeroes of infinite order), and $L_1$ (operating in $x\in\R^n$)
 and $L_2$ (operating in $y\in \R^m$) are general second order PDOs
 (as in \eqref{Lgenerale}) with smooth coefficients and they are assumed to be subelliptic.
 It is straightforward to recognize that by subtracting $\e$ to any PDO in
 \eqref{fediiOP}-to-\eqref{morimotoOP} we get an operator of the form
 $(L_1-\e)+\lambda(x)\,L_2$, where $\lambda$ has the required features, $L_2$ is uniformly elliptic (indeed, a classical Laplacian in all the
 examples), and $L_1-\e$ is
 a uniformly elliptic operator (cases \eqref{fediiOP}-to-\eqref{kusuokastroockOP})
 or it is a H\"ormander operator (case \eqref{morimotoOP}).
\end{remark}
  Before describing the approach to the Harnack Inequality, inspired by the ideas in \cite{Bony}, we
  state the main needed technical tools on the solvability of the Dirichlet problem
  for $\LL$ and for the perturbed operator $\LL-\varepsilon$.
\begin{lemma}\label{th.localDiri}
 Suppose that $\LL$ is an operator of the form
 \eqref{mainLL}, with $C^\infty$ coefficients $V>0$ and $(a_{i,j})\geq 0$,
 and that $\LL$ satisfies \emph{(NTD)}.
 Let $\e\geq 0$ be fixed (the case $\e=0$ being admissible).
 We set $\LL_\e:=\LL-\e$ and we assume that $\LL_\e$ is hypoelliptic on every open
 subset of $\RN$.\medskip

 Then, there exists a basis for the Euclidean topology of $\RN$, independent of $\e$,
 made of open and connected sets
 $\Omega$ (with Lipschitz boundary) with the following properties:
 for every continuous function $f$ on $\overline{\Omega}$ and for every continuous
 function $\varphi$ on $\de\Omega$, there exists one and only one solution
 $u\in C(\overline{\Omega},\R)$ of the Dirichlet problem
 \begin{gather}\label{DIRIEQ}
    \left\{
      \begin{array}{ll}
        \LL_\e u=-f & \hbox{on $\Omega$\quad  (in the weak sense of distributions),} \\
        u=\varphi & \hbox{on $\de\Omega$\quad (point-wise).}
      \end{array}
    \right.
 \end{gather}
 Furthermore, if $f,\varphi\geq 0$ then $u\geq 0$ as well. Finally, if $f$ belongs to $C^\infty(\Omega,\R)
 \cap C(\overline{\Omega},\R)$, then the same is true of $u$, and $u$
is a classical solution of \eqref{DIRIEQ}.
\end{lemma}
 We prove this theorem for a considerably larger
 class of operators than the $\LL_\e$ above; see Theorem \ref{th.localDiriMIGLIO}.
 We adapt to our context the well established techniques in \cite[Section 5]{Bony} used for H\"ormander operators.
 These techniques are perfectly suited to our more general case, since
 they only rely on hypoellipticity and on the Weak Maximum Principle.
 Since the proof presents no further difficulties, it is provided in the Appendix, for the sake of completeness only.

 With the existence of the weak solution of the Dirichlet problem for $\LL_\e$ on a bounded open set
 $\Omega$, we can define the associated Green operator as usual:
\begin{definition}[\textbf{Green operator and Green measure}]\label{defi.Green}
 Let $\e\geq 0$ be fixed, and let $\LL_\e$ and $\Omega$ satisfy, respectively,
 the hypothesis and the thesis of Lemma \ref{th.localDiri}. We consider the operator (depending on $\LL_\e$ and $\Omega$; we avoid
 keeping track of the dependency on $\Omega$ in the notation)
 \begin{equation}\label{greenoperator}
    G_\e: C(\overline{\Omega},\R)\longrightarrow C(\overline{\Omega},\R)
 \end{equation}
  mapping $f\in C(\overline{\Omega},\R)$ into the function $G_\e(f)$
  which is the unique distributional solution $u$ in $C(\overline{\Omega},\R)$
  of the Dirichlet problem
 \begin{gather}\label{DIRIEQGreen}
    \left\{
      \begin{array}{ll}
        \LL_\e u=-f & \hbox{on $\Omega$\quad  (in the weak sense of distributions),} \\
        u=0 & \hbox{on $\de\Omega$\quad (point-wise).}
      \end{array}
    \right.
 \end{gather}
 We call $G_\e$ \emph{the Green operator related to $\LL_\e$ and to the open set $\Omega$}.

 By the Riesz Representation Theorem (which is applicable thanks to the monotonicity pro\-per\-ties in
 Lemma \ref{th.localDiri} with respect to the function $f$), for every $x\in\overline{\Omega}$ there exists
 a (nonnegative) Radon measure $\lambda_{x,\e}$ on $\overline{\Omega}$ such that
 \begin{gather}\label{DIRIEQkernel}
  G_\e(f)(x)=\int_{\overline{\Omega}} f(y)\,\d\lambda_{x,\e}(y),\quad \text{for
  every $f \in C(\overline{\Omega},\R)$.}
 \end{gather}
 We call $\lambda_{x,\e}$ \emph{the Green measure related to $\LL_\e$ (to the open set $\Omega$ and to the point $x$)}.
\end{definition}
 Let $\LL$ be as in \eqref{mainLL}; in the rest of the paper, we set once and for all
\begin{equation}\label{VVmu}
 \d \nu(x):=V(x)\,\d x,
\end{equation}
  that is, $\nu$ is the (Radon) measure on $\RN$ associated with the (positive)
  density $V$ in \eqref{mainLL}, absolutely continuous with respect to the Lebesgue measure on $\RN$. It is clear that the
  measure $\nu$ plays the following key r\^ole:
  \begin{equation}\label{autoagg}
    \int \varphi\,\LL\psi\,\d\nu=\int \psi\,\LL\varphi\,\d\nu,\quad\text{
    for every $\varphi,\psi\in C_0^\infty(\RN,\R)$,}
  \end{equation}
 thus making $\LL$ (formally) self-adjoint in the space $L^2(\R^N,\d\nu)$.
 We observe that (in general) our operators $\LL$ in \eqref{mainLL} are not
 \emph{classically} self-adjoint; indeed the classical adjoint operator $\LL^*$ of $\LL$ is related to $\LL$
 by the following identity (a consequence of \eqref{autoagg})
 \begin{equation}\label{autoaggBISSS}
 \LL^* u=V\,\LL(u/V),\quad \text{for every $u$ of class $C^2$.}
 \end{equation}
 The possibility of dealing with non-identically $1$ densities $V$
 (as in the case of Lie groups, see Example \ref{exa.Lie}-(a))
 makes it more convenient  to decompose the Green measure $\lambda_{x,\e}$
 with respect to $\nu$ in \eqref{VVmu}, rather than w.r.t.\,Lebesgue measure.
 Hence we prove the following:
\begin{theorem}[\textbf{Green kernel}]\label{th.greeniani}
 Suppose that $\LL$ is an operator of the form
 \eqref{mainLL}, with $C^\infty$ coefficients $V>0$ and $(a_{i,j})\geq 0$,
 and that $\LL$ satisfies \emph{(NTD)}.
 Let $\e\geq 0$ be fixed.
 We set $\LL_\e:=\LL-\e$ and we assume that $\LL_\e$ is hypoelliptic on every open
 subset of $\RN$.\medskip

 Let $\Omega$ be an open set as in Lemma \ref{th.localDiri}.
 If $G_\e$ and $\lambda_{x,\e}$ are as in Definition \ref{defi.Green},
 there exists a function $k_\e:\Omega\times \Omega\to \R$, smooth and positive
 out of the diagonal of $\Omega\times \Omega$, such that the following representation holds true:
\begin{equation}\label{gprororfondam}
   G_\e(f)(x)=
   \int_\Omega f(y)\,k_\e(x,y)\,\d\nu(y), \quad
   \text{for every $x\in\Omega$,}
\end{equation}
   and for every $f \in C(\overline{\Omega},\R)$.
  We call $k_\e$ \emph{the Green kernel related to $\LL_\e$ and to the open set $\Omega$}.\medskip

  Furthermore, we have the following properties:
\begin{enumerate}
  \item[\emph{(i)}] Symmetry of the Green kernel:
\begin{equation}\label{gprororfondam2}
   k_\e (x,y)=k_\e (y,x) \quad
   \text{for every $x,y\in\Omega$.}
\end{equation}

  \item[\emph{(ii)}]
  For every fixed $x\in \Omega$, the function $k_\e (x,\cdot)$
  is $\LL_\e$-harmonic in $\Omega\setminus\{x\}$; moreover
  $G_\e(\LL_\e\varphi)=-\varphi=\LL_\e(G_\e(\varphi))$ for any $\varphi\in C_0^\infty (\Omega,\R)$, that is
\begin{gather}\label{gprororfondam2PROPR2}
\begin{split}
 -\varphi(x)&=\int_\Omega \LL_\e\varphi(y)\,k_\e (x,y)\,\d\nu(y) \\
  &=\LL_\e\Big(\int_\Omega \varphi(y)\,k_\e (x,y)\,\d\nu(y)\Big),
  \qquad
   \text{for every $\varphi\in C_0^\infty (\Omega,\R)$.}
\end{split}
\end{gather}

  \item[\emph{(iii)}]
  For every fixed $x\in \Omega$, one has
\begin{equation}\label{gprororfondam2PROPR23}
   \lim_{y\to y_0} k_\e (x,y)=0\quad \text{for any $y_0\in\de\Omega$.}
\end{equation}


  \item[\emph{(iv)}]
  For every fixed $x\in \Omega$, the functions
 $k_\e (x,\cdot)=k_\e (\cdot,x)$ are in $L^1(\Omega)$, and $k_\e \in L^1(\Omega\times \Omega)$.
\end{enumerate}
\end{theorem}
\noindent The key ingredients in the proof of the above results are the following facts:\medskip
 \begin{itemize}
   \item the hypoellipticity of $\LL_\e$ (as assumed in the hypothesis)
  which will imply the hypoellipticity of
  the \emph{classical} adjoint of $\LL_\e$ (see Remark \ref{rem.hyopoLadj});
   \item the $C^\infty$-topology on the space of the $\LL_\e$-harmonic functions
   is the same as the $L^1_{\textrm{loc}}$-topology,
   another consequence of the hypoellipticity of $\LL_\e$ (Remark \ref{rem.topolocinfyt});
   \item the fact that $\LL$ is self-adjoint on $L^2(\RN,\d\nu)$ (see \eqref{autoagg}) so that the same is true of $\LL_\e$
   (this will be crucial in proving the symmetry of the Green kernel);
   \item the Strong Maximum Principle for the perturbed operator $\LL_\e=\LL-\e$, which we obtain as a consequence of
   our previous Strong Maximum Principle for $\LL$ in Theorem \ref{th:SMP} (see precisely Remark \ref{PMFancheconc}, where
    \emph{nonnegative} maxima are considered): this is a key step for the proof of the \emph{positivity} of $k_\e$;
   \item the Schwartz Kernel Theorem (used for the regularity of the Green kernel).\medskip
 \end{itemize}
 The difference with respect to the analogous result given in the framework of the H\"ormander operators in \cite[Théorème 6.1]{Bony}
 is the introduction of the relevant measure $\nu $ in the integral representation \eqref{gprororfondam};
 indeed, the symmetry property \eqref{gprororfondam2} of the kernel $k_\e$ is connected with
 the identity \eqref{autoagg}, which is not true (in general) if
 we consider Lebesgue measure instead of $\nu$.\medskip

 We are now ready to give the second main result of the paper:
\begin{theorem}[\textbf{Strong Harnack Inequality}]\label{lem.crustimabassoTHEOFORTE}
 Suppose that $\LL$ is an operator of the form
 \eqref{mainLL}, with $C^\infty$ coefficients $V>0$ and $(a_{i,j})\geq 0$, and
 suppose it satisfies hypotheses \emph{(NTD)}, \emph{(HY)}
 and \emph{(HY)$_\e$}.\medskip

 Then, for every connected open set $O\subseteq \RN$ and every
 compact subset $K$ of $O$,
 there exists a constant $M=M(\LL,O,K)\geq 1$ such that
\begin{equation}\label{HarnackdeboleEQ1forte}
 \sup_{ K} u \leq M\,\inf_K u,
\end{equation}
 for every nonnegative $\LL$-harmonic function $u$ in $O$.\medskip

 If $\LL$ is subelliptic or if it has $C^\omega$ coefficients, then
 assumption \emph{(HY)$_\e$} can be dropped.
\end{theorem}
 \noindent The last assertion follows from Remark \ref{rem.equivHYeps}.

 We now present the spine of the proof of Theorem \ref{lem.crustimabassoTHEOFORTE}.\medskip

 The main step towards the Strong Harnack Inequality
 is the following Theorem \ref{th.mokobre} from Potential Theory.
 A proof of a more general abstract version of this useful result,
 in the framework of axiomatic harmonic spaces,
 can be found in the survey notes \cite[pp.20--24]{Brelot} by Brelot,
 where this theorem is attributed to G. Mokobodzki.
 (See also a further improvement to harmonic spaces
 which are not necessarily second-countable, by Loeb and Walsh, \cite{LoebWalsh}).
 Instead of appealing to an abstract Potential-Theoretic statement,
 we prefer to formulate the result under the following more specific form.
\begin{theorem}\label{th.mokobre}
 Let $L$ be a second order linear PDO in $\RN$ with smooth coefficients. Suppose the following conditions are satisfied.
\begin{description}
  \item[(Regularity)] There exists a basis $\mathcal{B}$ for the Euclidean topology of $\RN$
  (consisting of bounded open sets)
  such that, for every $\Omega\in \mathcal{B}\setminus\{\varnothing\}$
 and for every $\varphi\in C(\de\Omega,\R)$, there exists a unique
 $L$-harmonic function $H_\varphi^\Omega\in C^2(\Omega)\cap C(\overline{\Omega})$
 solving the Dirichlet problem
 $$\left\{
     \begin{array}{ll}
       L u=0 & \hbox{in $\Omega$} \\
       u=\varphi & \hbox{on $\de\Omega$,}
     \end{array}
   \right.
 $$
  and satisfying $H_\varphi^\Omega\geq 0$ whenever $\varphi\geq0$.

  \item[(Weak Harnack Inequality)]
 For every connected open set $O\subseteq \RN$, every
 compact subset $K$ of $O$ and every $y_0\in O$,
 there exists a constant $C(y_0)=C(L,O,K,y_0)>0$ such that
\begin{equation*}
 \sup_{ K} u \leq C(y_0)\,u(y_0),
\end{equation*}
 for every nonnegative $L$-harmonic function $u$ in $O$.
\end{description}
 Then, the following \emph{Strong Harnack Inequality} for $L$ holds:
 for every connected open set $O$ and every
 compact subset $K$ of $O$
 there exists a constant $M=M(L,O,K)\geq 1$ such that
\begin{equation}\label{SHIvera}
 \sup_{K} u \leq M\,\inf_{K} u,
\end{equation}
 for every nonnegative $L$-harmonic function $u$ in $O$.
\end{theorem}
 See also Remark \ref{rem.precisiamo}
 for some equivalent assumptions that can replace
 the above (Weak Harnack Inequality) to get the Strong
 Harnack Inequality.
 The proof of Theorem \ref{th.mokobre} is given in Section \ref{sec:Harnackvera},
 starting from a result by Mokobodzki and Brelot in \cite[Chapter I]{Brelot}:
 in the latter it is shown that if the axioms (Regularity) and (Weak Harnack Inequality) are fulfilled then,
  for any connected open set $O\subseteq\RN$ and any
 $x_0\in O$, the set
 \begin{equation}\label{equicontinuosa}
    \Phi_{x_0}:=\Big\{h\in \mathcal{H}_L(O)\,:\,h\geq 0,\quad h(x_0)=1\Big\}
 \end{equation}
 is equicontinuous at $x_0$. The proof of this fact rests on
 some deep results of Functional Analysis concerning the family
 of the so-called harmonic measures $\{\mu^\Omega_x\}_{x\in \de\Omega}$
 related to $L$ (and to a regular set $\Omega$ for the Dirichlet problem), jointly with some basic properties of the harmonic sheaf
 associated with the operator $L$.

 As observed by Bony in \cite[Remarque 7.1, p.300]{Bony},
 the Strong Harnack Inequality classically relies on two-sided estimates of the ratios
 $h(x_1,\cdot)/h(x_2,\cdot)$, where $h(x,y)$ is the relevant Poisson kernel; these estimates were unavailable in the
 setting considered in \cite{Bony}, as they are (to the best of our knowledge) in our setting too.
 However, like in \cite{Bony}, the unavailability of these estimates can be overcome
 by the use of the Green kernel for the perturbed operator $\LL-\e$ and by the Strong Maximum Principle, as they jointly lead
 to the Weak Harnack Inequality.
 It is interesting to observe that, once the Weak Harnack Inequality
 is available, the equicontinuity of \eqref{equicontinuosa} (an equivalent version of the Strong Harnack Inequality)
 is derived by Mokobodzki and Brelot by the comparison (in the sense of measures) $\mu^\Omega_{x_1}\leq M\,\mu^\Omega_{x_2}$ for harmonic measures:
 this comparison seems to be the core substitute for the mentioned pointwise estimates with Poisson kernels centered at different points $x_1,x_2$.\medskip

 Due to Theorem \ref{th.mokobre}, the focus on the Strong Harnack Inequality is now shifted to the Weak Harnack Inequality, which is easier to establish.
 As already anticipated, the latter is based on the lower bound
 \eqref{lowboundkeps} as we now briefly describe.
 First, we remark that the proof of \eqref{lowboundkeps} is a two-line comparison argument:
 it suffices to apply $\LL-\e$ on both sides of \eqref{lowboundkeps} to see that they produce the same result, namely $-\e\,u$;
 then one uses the Weak Maximum Principle, since the right-hand side is null on $\de\Omega$ whereas the left-hand side
 is nonnegative. Secondly, with inequality \eqref{lowboundkeps} at hands and the \emph{strict positivity} of $k_\e$ (a consequence
 of the SMP), it is not difficult to prove that $u(x_0)$ dominates the $L^1_\textrm{loc}$-norm of $u$, on suitable compact sets.
 Then, due to the equivalence of the $L^1_{\textrm{loc}}$ and $C^\infty$ topologies
 on the space of the $\LL$-harmonic functions (this fact deriving from (HY)), one can infer the following:
\begin{theorem}[Weak Harnack inequality for derivatives]\label{lem.crustimabassoTHEO}
 Let $\LL$ satisfy \emph{(NTD)}, \emph{(HY)}
 and \emph{(HY)}$_\e$.
 Then, for every connected open set $O\subseteq \RN$, every
 compact subset $K$ of $O$, every $m\in \mathbb{N}\cup\{0\}$ and every $y_0\in O$,
 there exists a positive $C(y_0)=C(\LL,\e,O,K,m,y_0)$ such that
\begin{equation}\label{HarnackdeboleEQ1}
 \sum_{|\alpha|\leq m}
 \sup_{x\in K} \Big| \frac{\de^\alpha u(x)}{\de x^\alpha}\Big|
\leq C(y_0)\,u(y_0),
\end{equation}
 for every nonnegative $\LL$-harmonic function $u$ in $O$.
\end{theorem}
 We remark that topological properties similar to those
 mentioned above for the space of the $\LL$-harmonic functions are also valid when $\LL$ in \eqref{mainLL}
 is \emph{not necessarily hypoelliptic}, provided that it possesses a
 positive global fundamental solution: see e.g., \cite{BattagliaBonfiglioli} by
 the first and third named authors, where Montel-type results are proved (in the sense of \cite{Montel}), jointly with the equivalence of the topologies
 induced on $\mathcal{H}_\LL(\Omega)$ by $L^1_{\textrm{loc}}$ and by $L^\infty_{\textrm{loc}}$,
under no hypoellipticity assumptions.\medskip

 \textbf{Acknowledgements.}
 We wish to thank Alberto Parmeggiani for many helpful discussions on hypoellipticty,
 leading to an improvement of the manuscript.

\section{The Strong Maximum Principle for $\LL$}\label{sec:SMP} 
 The aim of this section is to prove the Strong Maximum Principle for $\LL$
 in Theorem \ref{th:SMP}. Clearly, a fundamental step is played by a suitable Hopf-type lemma, furnished
 in Lemma \ref{lem_Hopf}.
 (For a recent interesting survey on maximum principles and Hopf-type results for 
 uniformly elliptic operators, see L\'opez-G\'omez \cite{lopez-Gomez}.)

 First the relevant definition and notation: given an open set $\Omega\subseteq \RN$
 and a relatively closed set $F$ in $\Omega$, we say that $\nu$ is \emph{externally orthogonal to $F$ at $y$}, and we write
 \begin{equation}\label{estnorm}
  \textrm{$\nu\bot\,F$ at $y$,}
 \end{equation}
 if: $y\in \Omega \cap \de F$; $\nu\in \RN\setminus\{0\}$; $\overline{B(y+\nu,|\nu|)}$
 is contained in $\Omega$ and it intersects $F$ only at $y$. Here and throughout $B(x_0,r)$ is the Euclidean
 ball in $\R^N$ of centre $x_0$ and radius $r> 0$; moreover $|\cdot|$ will denote the Euclidean norm on $\RN$.
 The notation \eqref{estnorm} does not explicitly refer to externality, 
 but this will not create any confusion in the sequel. It is not difficult to recognize that if $\Omega$ is connected and if $F$
 is a proper (relatively closed) subset of $\Omega$, then there always exist couples $(y,\nu)$
 such that $\nu\bot\,F$ at $y$.

 Finally, throughout the paper we write $\de_i$ for $\dfrac{\de}{\de x_i}$.
\begin{lemma}[\textbf{of Hopf-type for $\LL$}]\label{lem_Hopf}
 Suppose that $\LL$ is an operator of the form
 \eqref{mainLL} with $C^1$ coefficients $V>0$ and $a_{i,j}$,
 and let us set $A(x):=(a_{i,j}(x))_{i,j}$.
 (We recall that $A(x)\geq 0$ for every $x\in\RN$.)
 Let
 $\Omega\subseteq\RN$ be a connected open set. Then,
 the following facts hold.
\begin{enumerate}
  \item[\emph{(1)}]  Let $u\in C^2(\Omega,\R)$ be such that $\LL u\geq 0$ on $\Omega$;
  let us suppose that
 \begin{equation}\label{Hopf.ortFUU}
 F(u):=\Big\{x\in \Omega :\,u(x)=\max_\Omega u \Big\}
 \end{equation}
  is a proper subset of $\Omega$.
 Then
 \begin{equation}\label{Hopf.ort}
    \langle A(y)\nu,\nu\rangle=0\quad
    \text{whenever $\nu\bot\,F(u)$ at $y$}.
 \end{equation}

  \item[\emph{(2)}]
  Suppose $c\in C(\RN,\R)$ is nonnegative on $\RN$, and let us set
 $\LL_c:=\LL-c$. Let $u\in C^2(\Omega,\R)$ be such that $\LL_c u\geq 0$ on $\Omega$;
  let us suppose that $F(u)$ in \eqref{Hopf.ortFUU}
  is a proper subset of $\Omega$ and that $\max_\Omega u\geq 0$.
 Then \eqref{Hopf.ort} holds true.
\end{enumerate}
\end{lemma}
\begin{proof}
 We begin by proving part (1) in the statement of the Lemma, from which we also inherit
 the notation and hypotheses on $u$ and $F(u)$. Notice that  the assumption ensures that
 $M:=\max_\Omega u\in \R$.
 To this aim, let us assume by contradiction that
 \begin{equation} \label{lem_Hopf.EQ1}
  \langle A(y)\nu,\nu\rangle > 0, \quad \text{for some $\nu\bot\,F(u)$ at $y$.}
 \end{equation}
  We define a smooth function $w: \RN \longto \R$ as follows
 $$w(x) := e^{-\lambda|x - (y+\nu)|^2} - e^{-\lambda|\nu|^2},$$
 where $\lambda$ is a positive real number chosen in a moment.
 We set $b_j:=\sum_{i=1}^N\de_i(V\,a_{i,j})/V$, so that
 $\LL=\sum_{i,j}a_{i,j}\de_{i,j}+\sum_j b_j\de_j$.
 A simple computation shows that
 \begin{equation} \label{lem_Hopf.EQ1ccc}
 \LL w(y) = \lambda^2e^{-\lambda|\nu|^2}\left(4\langle A(y)\nu,\nu\rangle
 - \frac{2}{\lambda}\sum_{j = 1}^N\big(a_{j,j}(y) - b_j(y)\nu_j\big)\right),
 \end{equation}
  and thus, by \eqref{lem_Hopf.EQ1}, it is possible to choose $\lambda > 0$
 in such a way that $\LL w(y) > 0.$

 By the continuity of $\LL w$, we can then find a positive real number
 $\delta$ such that $V:= B(y,\delta)$ is compactly contained in
 $\Omega$ and $\LL w > 0$ on $V$.
 We now define, for $\e > 0$,
 a function $v_{\e}: \overline{V}\to\R$ by setting
 $v_{\e}:= u + \e \,w$.
 Clearly, $v_{\e} \in C^2(V,\R)\cap C(\overline{V},\R)$, and we claim that
 the maximum of $v_\e$ on $\overline{V}$ is attained in $V$.

 Indeed, let us consider the splitting of $\partial V$ given by the two sets
 $K_1 := \partial V\cap \overline{B(y+\nu,|\nu|)}$ and $K_2:=\partial V\setminus K_1.$
 For every $x \in K_2$, one has
 $$v_\e(x) = u(x) + \e w(x) < u(x) \leq M.$$
 On the other hand, for all $x \in K_1$,
 we have
 $$v_\e(x) \leq \max_{K_1}u + \e\max_{K_1}w,$$
 and since $\max_{K_1}u < M$
 (observe that $u<M$ outside $F(u)$ and that
 $K_1$ is a compact set contained in $\Omega \setminus F(u)$),
 it is possible to choose $\e>0$ so small that
 $v_\e < M$ on $K_1$.
 By gathering together these facts we see that,
 for every $x\in \de V$ (note that $y \in F(u)$ and $w(y) = 0$)
 $$v_\e(x) < M = u(y) = v_\e(y) \leq \max_{\overline{V}}v_\e, $$
 and this proves the claim.
 From $\LL v_\e=\LL u+\e\,\LL w\geq \e\,\LL w$ (and the latter is $>0$ on $V$) the function $v_\e$ is a
 \emph{strictly} $\LL$-subharmonic function on $V$, that is, $\LL v_\e > 0$ on $V$, admitting
 a maximum point on the open set $V$, say $p_0$.
 Then we have (recall that $A(p_0)\geq 0$ and notice that
 $\nabla v_\e(p_0)=0$ and  $H(p_0):=(\de_{i,j} {v_\e}(p_0))_{i,j} \leq 0$)
 \begin{align}\label{fine.hopf}
  0<\LL v_\e(p_0) =
  \sum_{i,j}a_{i,j}(p_0)\de_{i,j}v_\e(p_0)=
  \text{trace}\big(A(p_0)\cdot H(p_0)\big)
  \leq 0,
 \end{align}
 which is clearly a contradiction.
 \medskip

 Part (2) in the statement of the Lemma can be proved in a totally analogous way:
 we replace $\LL$ with $\LL_c$ and we
 notice that $w(y)=0$ so that $\LL_c w(y)=\LL w(y)$, and
 \eqref{lem_Hopf.EQ1ccc} is left unchanged.
 Arguing as above, we let again $p_0\in V$ be such that $v_\e(p_0)=\max_{\overline{V}} v_\e$.
 This gives $v_\e(p_0)\geq v_\e(y)=u(y)=M$. Hence \eqref{fine.hopf} becomes
 \begin{align*}
  0<\LL_c v_\e(p_0) =
  \text{trace}\big(A(p_0)\cdot H(p_0)\big) -c(p_0)\,v_\e(p_0)
  \leq -c(p_0)\,M,
 \end{align*}
 where in the last inequality we used the assumption $c\geq0$ and the fact that
 $v_\e(p_0)\geq M$.
 By the assumption $M\geq 0$ (and again by the assumption on the sign of $c$), we have 
 $-c(p_0)\,M\leq 0$, and we obtain another contradiction.
\end{proof}
 We are now in a position to provide the
\begin{proof}[Proof (of Theorem \ref{th:SMP})]
 Let $\LL$ be as in the statement of Theorem \ref{th:SMP};
 suppose that $\Omega\subseteq\RN$ is a connected open set
 and that $u\in C^2(\Omega,\R)$ satisfies $\LL u\geq 0$ on $\Omega$
 and $u$ attains a maximum in $\Omega$. We set
 $$F(u):=\Big\{x\in \Omega : \, u(x)=\max_\Omega u \Big\}.$$
 By assumption $F(u)\neq \varnothing$, say $\xi\in F(u)$. We show that $F(u)=\Omega$.

 To this aim, let us rewrite $\LL$ as follows:
 $$\LL=
 \frac{1}{V}\sum_{i,j} \de_i \Big(V\,a_{i,j}\,\de_j\Big)=
 \frac{1}{V} \sum_{i,j}  V\,\de_i( a_{i,j}\de_j)+
 \sum_{i,j} \frac{\de_i V}{V}\, a_{i,j}\de_j=
 \sum_{i,j}  \de_i( a_{i,j}\de_j)+
 \sum_{j} b_j\,\de_j,
 $$
 where $b_j:=\frac{1}{V}\sum_{i=1}^N\de_i V\, a_{i,j}$ (for $j=1,\ldots,N$).
 Let us consider the vector fields
 \begin{align}\label{CampiXX}
   X_i:= \sum_{j=1}^N  a_{i,j}\,\de_j,\quad i=1,\ldots,N,\qquad
   X_0:=\sum_{j=1}^N b_j\,\de_j.
 \end{align}
 We explicitly remak the following useful fact: $X_0$
 is a linear combination (with smooth-coefficients) of $X_1,\ldots,X_N$; indeed
\begin{equation}\label{X0lincopmbxi}
 X_0=\sum_{j=1}^N b_j\,\de_j=\sum_{j=1}^N \frac{1}{V}\sum_{i=1}^N\de_i V\, a_{i,j}\,\de_j
  =\sum_{i=1}^N\frac{\de_i V}{V}\,\sum_{j=1}^N  a_{i,j}\,\de_j =\sum_{i=1}^N\frac{\de_i V}{V}\,X_i.
\end{equation}
 Summing up, we have written $\LL$ as follows
 $$\LL u =\sum_{i=1}^N\de_i(X_iu)+\sum_{i=1}^N\frac{\de_i V}{V}\,X_iu,\quad \forall\,\,u\in C^2. $$
 Thanks to the assumption (NTD) of non-total degeneracy of $\LL$
   and due to the smoothness of its coefficients, we are entitled to use
   a notable result \cite[Theorem 2]{Amano} by Amano, which states that
  the hypoellipticity assumption (HY) ensures the controllability of the ODE system
\begin{equation}\label{X0lincopmbxiamano}
 \dot \gamma= \xi_0 X_0(\gamma)+\sum_{i=1}^N \xi_i\,X_i(\gamma),\quad (\xi_0,\xi_1,\ldots,\xi_N)\in \R^{1+N},
\end{equation}
 on every open and connected subset of $\RN$
 (see e.g., \cite[Chapter 3]{Jurdjevic} for the notion of controllability).
 Since $\Omega$ is open and connected, this implies that
 any point of $\Omega$ can be joined to $\xi$ by a continuous curve $\gamma$
 contained in $\Omega$ which is piecewise an integral curve of
 a vector field belonging to $\mathcal{V}:=\textrm{span}_\R\{X_0,X_1,\ldots,X_N\}$.
 It then suffices to prove that if $\gamma$ is
 an integral curve of a vector field $X\in \mathcal{V}$ starting at a point of $F(u)$ (which is non-empty), then
 $\gamma(t)$ remains in $F(u)$ for every admissible time $t$. In this case
 we say that $F(u)$ is \emph{$X$-invariant}.

 By a result of Bony, \cite[Théorème 2.1]{Bony}, the $X$-invariance of $F(u)$
 is equivalent to the tangentiality of $X$ to $F(u)$: this latter condition means that
\begin{equation}\label{daprovSMP1primaa}
 \lan X(y),\nu\ran =0\quad \text{whenever $\nu\bot\,F(u)$ at $y$}.
\end{equation}
 Hence, by all the above arguments, the proof of the SMP is complete if we show that 
 \eqref{daprovSMP1primaa} is fulfilled by any $X\in \mathcal{V}$.
  Since $X$ is a linear combination of $X_0,X_1,\ldots,X_N$ and due to \eqref{X0lincopmbxi}, it suffices to prove
  this identity when $X$ is replaced by any element of $\{X_1,\ldots,X_N\}$.
  Due to identity  \eqref{Hopf.ort} in the Hopf-type Lemma \ref{lem_Hopf}, it is therefore
  sufficient to show that for every $i\in\{1,\ldots,N\}$ and every $x\in \RN$, there exists
  $\lambda_i(x)>0$ such that
  \begin{equation}\label{daprovSMP1}
   \lan X_i(x),\nu\ran^2 \leq \lambda_i(x)\,\langle A(x)\nu,\nu\rangle\quad
    \text{for every $\nu\in\RN$}.
  \end{equation}
 Indeed, \eqref{daprovSMP1} together with \eqref{Hopf.ort} implies that 
 the left-hand side of \eqref{daprovSMP1} is null whenever $\nu\bot F(u)$ at $y$, which is precisely 
 \eqref{daprovSMP1primaa} for $X\in\{X_1,\ldots,X_N\}$.
 Due to the very definition of $X_i$, inequality \eqref{daprovSMP1} boils down to proving that, given
 a real symmetric positive semidefinite matrix $A=(a_{i,j})$, for every $i$ there exists $\lambda_i>0$ such that
 $$\Big(\sum_{j}a_{i,j}\nu_j\Big)^2\leq \lambda_i\,\sum_{i,j} a_{i,j}\,\nu_i\,\nu_j
 \quad \text{for every $\nu\in\RN$},$$
 which is a consequence of the Cauchy-Schwarz inequality and the characterization
 (for a symmetric $A\geq 0$)
 of $\textrm{ker}(A)$ as $\{x\in\RN:\lan Ax,x\ran=0\}$.\medskip

 This proves part (1) of Theorem \ref{th:SMP}. As for part (2), let
 $c$ be smooth and nonnegative on $\RN$, and let us set $\LL_c:=\LL-c$.
 Suppose $u\in C^2(\Omega,\R)$ satisfies $\LL_c u\geq 0$ on $\Omega$
 and that it attains a nonnegative maximum in $\Omega$. For $F(u)\neq \varnothing$ as above, we show again that
 $F(u)=\Omega$. The hypoellipticity and non-total degeneracy of $\LL$
 ensure again (by Amano's cited result for $\LL$) the controllability of
 system \eqref{X0lincopmbxiamano}. This again grants a connectivity property of $\Omega$
 by means of continuous curves, piecewise integral curves of elements in the above vector space $\mathcal{V}$.
 By Bony's quoted result on invariance/tangentiality, the needed identity $F(u)=\Omega$ follows if we show again that
  \eqref{daprovSMP1primaa} is fulfilled when $X$ is replaced by $X_i$, for $i=1,\ldots,N$ (the case $i=0$ deriving as above from
  \eqref{X0lincopmbxi}).

 Now, by part (2) of Lemma \ref{lem_Hopf},
 it is at our disposal a Hopf-type Lemma for operators
 of the form $\LL_c$, and for functions $u$ such that $\LL_c u\geq 0$
 and attaining a \emph{nonnegative} maximum. In other words, we know that
 \eqref{Hopf.ort} holds true, again as in the previous case (1).  The validity of \eqref{daprovSMP1} allows us to end the proof,
 as in the previous part.
\end{proof}
 A close inspection to the above proof shows that we have indeed
 demonstrated the following result as well (replacing the hypothesis of hypoellipticity 
 of $\LL$ by that of $\LL-c$), since Amano's results on hypoellipticity/controllobility
 are independent of the presence of a zero-order term:
\begin{remark}\label{PMFancheconc}
 \emph{Suppose that $\LL$ is an operator of the form
 \eqref{mainLL}, with $C^\infty$ coefficients $V>0$ and $(a_{i,j})\geq 0$,
 and that it satisfies \emph{(NTD)}.
 Let $c\in C^\infty(\RN,\R)$ be nonnegative and
 suppose that the operator $\LL_c:=\LL-c$ is hypoelliptic
 on every open subset of $\RN$.}

 \emph{If
 $\Omega\subseteq\RN$ is a connected open set, then any function
 $u\in C^2(\Omega,\R)$ satisfying $\LL_c u\geq 0$ on $\Omega$ and
 attaining a nonnegative maximum in $\Omega$ is constant
 throughout $\Omega$.}
\end{remark}
 As a Corollary of Theorem \ref{th:SMP} we immediately get the following result.
\begin{corollary}[\textbf{Weak Maximum Principle for $\LL$}]\label{th.WMPPP}
  Suppose that $\LL$ is an operator of the form
 \eqref{mainLL}, with $C^\infty$ coefficients $V>0$ and $(a_{i,j})\geq 0$,
 and that it satisfies \emph{(NTD)} and \emph{(HY)}. Suppose also that
 $c\in C^\infty(\RN,\R)$ is nonnegative on $\RN$ (the case $c\equiv 0$ is allowed),
 and let us set $\LL_c:=\LL-c$.
 Then,
 $\LL_c$ satisfies the \emph{Weak Maximum Principle} 
 on every bounded open set
   $\Omega\subseteq\RN$, that is:
\begin{equation}\label{WMP}
 \left\{
   \begin{array}{ll}
     u\in C^2(\Omega,\R)\\
     \LL_c u\geq 0\,\,\text{on $\Omega$} \\
     \limsup\limits_{x\to x_0} u(x)\leq 0\,\,\text{for every $x_0\in\de\Omega$}
   \end{array}
 \right.
 \qquad\Longrightarrow\qquad
 u\leq 0\,\,\text{on $\Omega$.}
\end{equation}
 As a consequence, if $\Omega\subseteq\RN$ is bounded, and if
 $u\in C^2(\Omega)\cap C(\overline{\Omega})$ is nonnegative and
 such that $\LL_c u\geq 0$ on $\Omega$, then one has
 $ \sup_{\overline{\Omega}} u=\sup_{\de\Omega} u. $
\end{corollary}
\begin{proof}
 Suppose that the open set $\Omega\subset\RN$ is bounded and $u$ is as in the left-hand side of
 \eqref{WMP}. Let $x_0\in \overline{\Omega}$ be such that
\begin{equation}\label{weier.max}
 \limsup_{x\to x_0} u(x)=\sup_{\Omega}u.
\end{equation}
 If $x_0\in\de\Omega$, then \eqref{WMP} ensures that $\limsup_{x\to x_0} u(x)\leq 0$, so that (due to
 \eqref{weier.max}) $\sup_\Omega u\leq0$, proving the right-hand side of \eqref{WMP}.
 If $x_0\in\Omega$, then \eqref{weier.max} gives $u(x_0)=\max_\Omega u$.
 If $u(x_0)<0$, we conclude as above that $\max_\Omega u=u(x_0)<0$. If $u(x_0)\geq 0$,
 we consider $\Omega_0\subseteq\Omega$ the connected component of $\Omega$ containing $x_0$, and,
 thanks to part (2) of the Strong Maximum Principle in Theorem \ref{th:SMP},
 the existence of an interior maximum point of $u$ on $\Omega \supseteq \Omega_0$
 (and the fact that $u(x_0)\geq 0$)
 ensures that $u\equiv u(x_0)$ on $\Omega_0$. Let us take any $\xi_0\in \de \Omega_0$; we have
 $$\max_{\Omega}u=u(x_0)=\limsup_{\Omega_0\ni x\to \xi_0}u(x)\leq
 \limsup_{\Omega\ni x\to \xi_0}u(x)\leq  0,  $$
 where the last inequality follows from $\de \Omega_0\subseteq \de\Omega$
 and from the assumption in \eqref{WMP}.

 We remark that when $c\equiv 0$ the proof is slightly simpler, as an interior maximum of
 $u$ propagates up to the boundary, regardless of the sign of this maximum.

 Finally we prove the last assertion of the corollary.
 Let $\Omega\subseteq\RN$ be bounded and let $u\in C^2(\Omega)\cap C(\overline{\Omega})$ 
 be nonnegative on $\overline{\Omega}$ and satisfying $\LL_c u\geq 0$ on $\Omega$; then we set $M:=\sup_{\de\Omega} u$
 and we observe that $M\geq 0$ since this is true of $u$. 
 We have (recall that $c\geq 0$)
 $$\LL_c(u-M)=\LL_c u-\LL_c M\geq -\LL_c M=-(\LL -c)M=cM\geq 0.  $$
 Since (by definition of $M$) we have $u-M\leq 0$ on $\de\Omega$ (and $u-M$ is continuous up to $\de\Omega$),
 we can apply \eqref{WMP} to get $u-M\leq 0$, that is $u\leq \sup_{\de\Omega} u$ on $\Omega$.
 This clearly proves the needed $ \sup_{\overline{\Omega}} u=\sup_{\de\Omega} u $.
\end{proof}
 Arguing as in the previous proof (and exploiting Remark \ref{PMFancheconc}
 instead of Theorem \ref{th:SMP}-(2)) we also get the following result,
 where we alternatively replace the hypothesis of hypoellipticity
 of $\LL$ by that of $\LL-c$:
\begin{remark}\label{PMFancheconc2}
 \emph{Suppose that $\LL$ is an operator of the form
 \eqref{mainLL}, with $C^\infty$ coefficients $V>0$ and $(a_{i,j})\geq 0$,
 and that it satisfies \emph{(NTD)}.
 Let $c\in C^\infty(\RN,\R)$ be nonnegative and
 suppose that the operator $\LL_c:=\LL-c$ is hypoelliptic
 on every open subset of $\RN$.}

 \emph{Then $\LL_c$ satisfies the Weak Maximum Principle
 on every bounded open set $\Omega\subseteq\RN$.}

 \emph{As a consequence, if $\Omega\subseteq\RN$ is bounded, and if
 $u\in C^2(\Omega)\cap C(\overline{\Omega})$ is nonnegative and
 such that $\LL_c u\geq 0$ on $\Omega$, then one has}
 $ \sup_{\overline{\Omega}} u=\sup_{\de\Omega} u. $
\end{remark}
\section{Analytic coefficients: A Unique Continuation result for $\LL$}\label{sec:Harnack_analytic}
 In this short section, by means of the ideas of controllability/propagation introduced in the previous section,
 we prove the following result.
\begin{theorem}[\textbf{Unique Continuation for $\LL$}]\label{th:Uniquecont}
 Suppose that $\LL$ is an operator of the form
 \eqref{mainLL} satisfying assumptions \emph{(NTD)} and \emph{(HY)}.
 Suppose that $\LL$ has $C^\omega$ coefficients $a_{i,j}$ and $V$.

  Let
 $\Omega\subseteq\RN$ be a connected open set.
 Then any $\LL$-harmonic function
 on $\Omega$ vanishing on some non-empty open subset of $\Omega$ is identically zero on $\Omega$.
\end{theorem}
\begin{proof}
 Let $u\in \mathcal{H}_\LL(\Omega)$ be vanishing on the open set $U\subseteq \Omega$, $U\neq\varnothing$.
 Let $F\subseteq \Omega$ be the support of $u$.
 We argue by contradiction, by assuming that $F\neq \varnothing$.
 Let us fix any $y\in \de F\cap \Omega$
 and any $\nu\in \RN\setminus\{0\}$ such that $\nu \bot F$ at $y$ (see the notion of exterior
 orthogonality at the beginning of Section \ref{sec:SMP}; the assumption $F\neq \varnothing$
 ensures the existence of such a couple $(y,\nu)$). We consider the Euclidean open ball
 $B:=B(y+\nu,|\nu|)$ which is completely contained in $\Omega\setminus F$, so that $u\equiv 0$ on $B$.
 We observe that $B$
 is the sub-level set $\{f(x)<0\}$, where
 $$f(x)=|x-(y+\nu)|^2-|\nu|^2. $$
 There are only two cases:\medskip

 (a) The boundary of $B$ is \emph{non-characteristic} for $\LL$ at $y$, that is,
 $\lan A(y)\,\nabla f(y),\nabla f(y)\ran \neq 0$.
 Due to the $C^\omega$ assumption we are allowed to use the classical Holmgren's Theorem
 (see e.g., \cite[Theorem 8.6.5]{HormanderLibro}),
 ensuring that
 $u$ vanishes in a neighborhood of $y$, so that $y\in\Omega\setminus F$.
 Since $F$ is relatively closed in $\Omega$, this is in contradiction with $y\in \de F\cap \Omega$.
 Hence it is true that:\medskip

 (b) The boundary of $B$ is \emph{characteristic} for $\LL$ at $y$, that is,
 $\lan A(y)\,\nabla f(y),\nabla f(y)\ran = 0$.
 Since $\nabla f(x)=2\,(x-y-\nu)$, this condition boils down to
 $\lan A(y)\nu,\nu\ran=0$.
 Let $X_0,X_1,\ldots,X_N$ be the vector fields introduced
 in \eqref{CampiXX}.
 The same Linear Algebra argument
 leading to \eqref{daprovSMP1} shows that
 $\lan A(y)\nu,\nu\ran=0$ implies
 $$\lan X_i(y),\nu\ran=0\quad \text{for every $i=1,\ldots,N$.} $$
 Identity \eqref{X0lincopmbxi} guarantees that the same holds for $i=0$ as well.
 Therefore, one has  $\lan X(y),\nu\ran=0$ for every $X\in \mathcal{V}:=\textrm{span}\{X_0,X_1,\ldots,X_N\}$.
 Arguing as in Section \ref{sec:SMP}, by means of the result by Bony \cite[Théorème 2.1]{Bony},
 this geometric condition (holding true for arbitrary $\nu \bot F$ at $y$) implies that the closed set $F$ is $X$-invariant
 for any $X\in \mathcal{V}$ (that is $F$ contains the trajectories of the integral curves
 of $X$ touching $F$).

 On the other hand, the hypoellipticity assumption (HY) on $\LL$ ensures
 (due to the recalled result by Amano \cite[Theorem 2]{Amano}) that
 any pair of points of the connected open set $\Omega$ can be joined by a continuous curve which
 is piecewise an integral curve of some vector fields $X$ in  $\mathcal{V}$. Gathering together all the mentioned
 results, the fact that $F\neq \varnothing$ implies that any point of
 $\Omega$ belongs to $F$, contradicting the assumption that $U\subseteq \Omega\setminus F$.
\end{proof}

 \section{The Green function and the Green kernel for $\LL-\e$}\label{sec:Green}
 The aim of this section is to prove
 Theorem \ref{th.greeniani}. In the first part of the proof
 (Steps I--III) we follow the classical scheme by Bony (see \cite[Théorème 6.1]{Bony}),
 hence we skip many details; it is instead in Step IV that a slight difference is presented,
 in that we exploit the measure $\d\nu(x)=V(x)\,\d x$ in order to obtain the symmetry
 property of the Green kernel even when our operator $\LL$ is not (classically) self-adjoint.
 The problem of the behavior of the Green kernel along the diagonal is more subtle, as it is
 shown by Fabes, Jerison and Kenig in \cite{FabesJerisonKenig} who proved that, for divergence-form
 operators as in \eqref{mainLL} (when $V\equiv 1$ and, roughly put, when the degeneracy of $A(x)$
 is controlled by a suitable weight) the limit of the 
 Green kernel along the diagonal need not be infinite; we plan to investigate this behavior in a future study,
 since our assumption (NTD) prevents the existence of any vanishing Muckenhoupt-type weight.\medskip

 Throughout this section, we fix an operator $\LL$ of the form
 \eqref{mainLL}, with $C^\infty$ coefficients $V>0$ and $(a_{i,j})\geq 0$,
 and we assume that $\LL$ satisfies (NTD).
 Moreover, we also fix $\e\geq 0$ (note that the case $\e=0$ is allowed)
 and we set $\LL_\e:=\LL-\e$; we assume that $\LL_\e$ is hypoelliptic on every open
 subset of $\RN$. Finally, $\Omega$ is a fixed open set as
 in Lemma \ref{th.localDiri}, such that the Dirichlet problem
 \eqref{DIRIEQ} is (uniquely) solvable.

 From Lemma \ref{th.localDiri}, we know that there exists a monotone operator $G_\e$
 (which we called the Green operator related to $\LL_\e$ and $\Omega$); since $\e\geq0$ is fixed, in all this section 
 we drop the subscript $\e$ in $G_\e,k_\e,\lambda_{x,\e}$ and we simply write $G,k,\lambda_x$. Hence we are given the monotone operator
\begin{equation*}
    G: C(\overline{\Omega},\R)\longrightarrow C(\overline{\Omega},\R)
\end{equation*}
 mapping $f\in C(\overline{\Omega},\R)$ into the unique function $G(f)\in
 C(\overline{\Omega},\R)$
  satisfying
 \begin{gather}\label{DIRIEQGreenBIS}
    \left\{
      \begin{array}{ll}
        \LL_\e (G(f))=-f & \hbox{on $\Omega$\quad  (in the weak sense of distributions),} \\
        G(f)=0 & \hbox{on $\de\Omega$\quad (point-wise).}
      \end{array}
    \right.
 \end{gather}
 We also know that the (Riesz) representation
\begin{gather}\label{DIRIEQkernelBIS}
  G(f)(x)=\int_{\overline{\Omega}} f(y)\,\d\lambda_x(y)\quad \text{for
  every $f \in C(\overline{\Omega},\R)$ and every $x\in\overline{\Omega}$}
 \end{gather}
 holds true, with a unique Radon measure
 $\lambda_x$ defined on $\overline{\Omega}$ (which we called
 the Green measure related to $\LL_\e$, $\Omega$ and $x$).

 Finally, we set $\d \nu(x):=V(x)\,\d x$
 and we observe that (as in \eqref{autoagg})
  \begin{equation}\label{autoaggBIS}
    \int \varphi\,\LL_\e\psi\,\d\nu=\int \psi\,\LL_\e\varphi\,\d\nu,\quad\text{
    for every $\varphi,\psi\in C_0^\infty(\RN,\R)$.}
  \end{equation}

 \textsc{Step I.}
 We fix $x\in \Omega$.
 We begin by proving that $\lambda_x$
 is absolutely continuous with respect to the Lebesgue measure on $\overline{\Omega}$.
 To this end, let $\varphi\in C_0^\infty(\Omega,\R)$;
 by \eqref{DIRIEQGreenBIS} it is clear that $G(\LL_\e\varphi)=-\varphi$, so that (see
 \eqref{DIRIEQkernelBIS})
\begin{gather*}
  -\varphi(x)=\int_{\overline{\Omega}} \LL_\e\varphi(y)\,\d\lambda_x(y),\quad \text{for
  every $\varphi \in C_0^\infty(\overline{\Omega},\R)$.}
 \end{gather*}
 If we consider $\lambda_x$ as a distribution on $\Omega$ in the standard way, this identity
 boils down to
\begin{equation}\label{DIRIIII}
 (\LL_\e)^* \lambda_x=-\textrm{Dir}_x\quad \text{in $\mathcal{D}'(\Omega)$},
\end{equation}
 where $\textrm{Dir}_x$ denotes the Dirac mass at $x$, and $(\LL_\e)^*$
 is the classical adjoint operator of $\LL_\e$. It is noteworthy to observe
 that, in general, $(\LL_\e)^*$ is neither equal to $\LL_\e$ nor of the form $\widetilde{\LL}-\e$
 for any $\widetilde{\LL}$ a divergence operator as in \eqref{mainLL}.

 However, the following crucial property of $(\LL_\e)^*$ is fulfilled:
\begin{remark}\label{rem.hyopoLadj}
 \emph{The operator $(\LL_\e)^*$ is hypoelliptic on every open subset of $\RN$.}\medskip

 Indeed, let $U\subseteq W$ be open sets and let $u\in \mathcal{D}'(W)$ be
 such that $(\LL_\e)^* u=h$ in $\mathcal{D}'(U)$, where $h\in C^\infty(U,\R)$. This gives the following chain of identities
 (here $\psi\in C_0^\infty(U,\R)$ is arbitrary)
\begin{align*}
    \int h\,\psi&= \lan u,\LL_\e\psi\ran= \lan u,\LL\psi-\e\psi\ran
  \stackrel{\eqref{autoaggBISSS}}{=}
 \Big\lan u,\frac{\LL^*(V\psi)}{V}-\e\psi\Big\ran\\
& =
 \Big\lan \frac{u}{V},\LL^*(V\psi)-\e\psi\,V\Big\ran
=\Big\lan \frac{u}{V},(\LL_\e)^*(V\psi)\Big\ran.
\end{align*}
 If we write
 $\int h\,\psi=\int \frac{h}{V}\,(\psi\,V)$, and if we observe that
 $C_0^\infty(U,\R)=\{\psi\,V\,:\,\psi\in C_0^\infty(U,\R)\}$,
 the above computation shows that $\LL_\e(u/V)=h/V$ in $\mathcal{D}'(U)$.
 The hypoellipticity of $\LL_\e$ now gives $u/V\in C^\infty(U,R)$ whence
 $u\in C^\infty(U,R)$, as $V$ is smooth and positive.\qed
\end{remark}
  Identity \eqref{DIRIIII} gives in particular
 $(\LL_\e)^* \lambda_x= 0$ in $\mathcal{D}'(\Omega\setminus\{x\})$;
 thanks to Remark \ref{rem.hyopoLadj},
 this ensures the existence of $g_x\in C^\infty(\Omega\setminus\{x\},\R)$
 such that the distribution $\lambda_x$ restricted to
 $\Omega\setminus\{x\}$ is the function-type distribution associated
 with the function $g_x$; equivalently
\begin{equation}\label{cerntarel16volte}
    \int \varphi(y)\,\d\lambda_x(y) = \int \varphi(y)\,g_x(y)\,\d y,\quad
\text{for every $\varphi\in C_0^\infty (\Omega\setminus\{x\},\R)$.}
\end{equation}
 Clearly $g_x\geq 0$ on $\Omega\setminus\{x\}$
 and $(\LL_\e)^* g_x=0$ in $\Omega\setminus\{x\}$.
 This temporarily proves that $\lambda_x$ coincides with
 $g_x(y)\,\d y$ on $\Omega\setminus\{x\}$. We claim that this is also true
 throughout $\Omega$. This will follow if we show that $C:=\lambda_x(\{x\})=0$.
 Clearly, by the definition of $C$,  on $\Omega$ we have
 $$\lambda_x=C\,\textrm{Dir}_x+(\lambda_x)|_{\Omega\setminus\{x\}}=
 C\,\textrm{Dir}_x+g_x(y)\,\d y.$$
 Treating this as an identity between distributions on $\Omega$,
 we apply the operator $(\LL_\e)^*$ to get
 $$C\,(\LL_\e)^*\textrm{Dir}_x= -\textrm{Dir}_x-
 (\LL_\e)^*(g_x(y)\,\d y).$$
 Here we used \eqref{DIRIIII}.
 We now proceed as follows:
\begin{enumerate}[-]
  \item we multiply both sides by a $C^\infty$ function $\chi$
  compactly supported in $\Omega$ and $\chi\equiv 1$ near $x$;

  \item  we compute the Fourier transform of the tempered distributions
 obtained as above;

  \item on the left-hand side we obtain a function-type distribution
 associated with function
 $$y\mapsto C\,e^{-i\lan x,y\ran}\Big(-\sum_{i,j} a_{i,j}(x)\,y_iy_j
 +\{\text{polynomial in $y$ of degree $\leq 1$}\}\Big),$$
 where $(a_{i,j})$ is the principal matrix of $\LL$;

 \item on the right-hand side
 we obtain a function-type distribution
 associated with a function which is the sum of $y\mapsto -e^{-i\lan x,y\ran}$
 with a function of the form
 $$y \mapsto -\sum_{i,j} \alpha_{i,j}(x,y)\,y_iy_j
 +\{\text{polynomial in $y$ of degree $\leq 1$}\},$$
 where
 $$\alpha_{i,j}(x,y)=-\int g_x(\xi)\,\chi(\xi)\,a_{i,j}(\xi)\,
 e^{-i\lan \xi,y\ran}\,\d \xi.$$
 By the Riemann-Lebesgue Theorem one has $\alpha_{i,j}(x,y)\longto 0$ as $|y|\to \infty$.
 This implies that $C=0$, since at least one of the entries of
 $(a_{i,j}(x))$ is non-vanishing, due to the (NTD) hypothesis on $\LL$.
\end{enumerate}
We have therefore proved that,
 for any $x\in\Omega$,
\begin{equation}\label{suomegamiusss}
 \text{$\d\lambda_x(y)=g_x(y)\,\d y$ on $\Omega$.}
\end{equation}
 Since $\lambda_x$ is a finite measure (recalling that $\overline{\Omega}$ is compact),
 from \eqref{suomegamiusss} we get $g_x\in L^1(\Omega)$ for every
 $x\in \Omega$. \medskip

 \textsc{Step II.}
 We next show that $\lambda_x(\de \Omega)=0$ for any $x\in\overline{\Omega}$.
 For small $\delta>0$, we let $D_\delta$ denote the closed $\delta$-neighborhood
 of $\de\Omega$ of the points in $\RN$ having distance from $\de\Omega$
 less than or equal to $\delta$; we then choose a function $F\in C(\RN,[0,1])$ which is identically
 $1$ on $\de\Omega$ and is supported in the interior of $D_\delta$.
 We denote by $f$ the restriction of $F$ to $\overline{\Omega}$.
 From \eqref{DIRIEQkernelBIS} we have
\begin{gather}\label{DIRIEQkernelBISiii}
  0\leq  G(f)(x)=\int_{\overline{\Omega}} f(y)\,\d\lambda_x(y)
 \leq \int_{\overline{\Omega}} \d\lambda_x(y)=G(1)(x),
\quad \text{for every $x\in \overline{\Omega}$.}
 \end{gather}
  For any $x\in \overline{\Omega}$ we have
\begin{align*}
 \lambda_x(\de\Omega)&=\int_{\de\Omega}\d\lambda_x(y)=
 \int_{\de\Omega}f(y)\,\d\lambda_x(y)\leq \int_{\overline{\Omega}} f(y)\,\d\lambda_x(y)=
  G(f)(x)\\
 &\leq \sup_{\overline{\Omega}} G(f)
 =\max\bigg\{\sup_{\overline{\Omega}\cap {D_\delta}} G(f),
 \sup_{\overline{\Omega}\setminus D_\delta} G(f)\bigg\}
 =:\max\{\textrm{I},\textrm{II}\}.
\end{align*}
 We claim that $\textrm{I}$ and $\textrm{II}$ in the above right-hand side are
 bounded from above by $\sup_{\overline{\Omega}\cap D_\delta} G(1)$. This is true of $\textrm{I}$, due
 to \eqref{DIRIEQkernelBISiii}; as for $\textrm{II}$ we invoke
 the last assertion in Remark \ref{PMFancheconc2} applied to:
\begin{enumerate}[-]
  \item the hypoelliptic operator $\LL_\e=\LL-\e$,
  \item the bounded open set $\Omega_1:=\overline{\Omega}\setminus D_\delta$,
  \item the nonnegative function $G(f)$, which satisfies $\LL_\e G(f)=-f=0$ on $\Omega_1$ both weakly
  and strongly due to the hypoellipticity of $\LL_\e$.
\end{enumerate}
 The mentioned Remark \ref{PMFancheconc2} then ensures that the values of $G(f)$ on
 $\overline{\Omega}\setminus D_\delta$ are bounded from above by the values
 of $G(f)$ on the boundary of this set, so that $\textrm{II}\leq \textrm{I}$.
 Summing up,
\begin{align*}
 \lambda_x(\de\Omega)&\leq
 \max\{\textrm{I},\textrm{II}\}\leq
 \sup_{\overline{\Omega}\cap D_\delta} G(1).
\end{align*}
 As $\delta$ goes to $0$, the right-hand side tends to
  $\sup_{\de \Omega} G(1)=0$ by \eqref{DIRIEQGreenBIS}. This gives the desired
  $\lambda_x(\de\Omega)=0$, for any $x\in \overline{\Omega}$.
 By collecting together \eqref{suomegamiusss} and $\lambda_x(\de\Omega)=0$, we infer that (for
  every $f \in C(\overline{\Omega},\R)$ and $x\in {\Omega}$)
\begin{gather*}
  G(f)(x)\stackrel{\eqref{DIRIEQkernelBIS}}{=}
  \int_{\overline{\Omega}} f(y)\,\d\lambda_x(y)
 =
 \int_{\Omega} f(y)\,\d\lambda_x(y)
 \stackrel{\eqref{suomegamiusss}}{=}
  \int_{\Omega} f(y)\,g_x(y)\,\d y.
 \end{gather*}
 This proves the identity
\begin{gather}\label{DIRIEQkernelBISTERQ}
  G(f)(x)=\int_{\Omega} f(y)\,g_x(y)\,\d y,\quad \text{for
  every $f \in C(\overline{\Omega},\R)$ and every $x\in  \Omega$.}
\end{gather}
 If $\varphi\in C_0^\infty(\Omega,\R)$, since we know that $G(\LL_\e \varphi)=-\varphi$, we get
\begin{gather}\label{DIRIEQkernelBISTERQuater}
  -\varphi(x)=\int_{\Omega} \LL_\e\varphi(y)\,g_x(y)\,\d y,\quad \text{for
  every $x\in  \Omega$.}
\end{gather}
 This is equivalent to
\begin{gather}\label{DIRIEQkernelBISTERQuater55}
 (\LL_\e)^* g_x=-\text{Dir}_x\quad \text{for every $x\in\Omega$.}
\end{gather}

 \textsc{Step III.}
  If $g_x$ is as in Step I,
  we are ready to set
\begin{equation*}
  g:\Omega\times\Omega \longrightarrow [0,\infty],\qquad
 g(x,y):=\left\{
           \begin{array}{ll}
             g_x(y) & \hbox{if $x\neq y$} \\
             \infty & \hbox{if $x=y$.}
           \end{array}
         \right.
\end{equation*}
 Hence the representation
 \eqref{DIRIEQkernelBISTERQ}
 becomes
\begin{gather}\label{quasimoltk}
  G(f)(x)=\int_{\Omega} f(y)\,g(x,y)\,\d y,\quad \text{for
  every $f \in C(\overline{\Omega},\R)$ and every $x\in  \Omega$.}
\end{gather}
 We aim to prove that $g$ is smooth outside the diagonal of $\Omega\times\Omega$.
\begin{remark}\label{rem.topolocinfyt}
 Let $O$ be any open subset of $\RN$.
 \emph{The hypoellipticity of a general PDO $L$ as in \eqref{Lgenerale} ensures the equality of the topologies
 on $\H_{L}(O)$ inherited by the Fréchet spaces $C^\infty(O)$
 and $L^1_{\mathrm{loc}}(O)$.}\medskip

 Indeed, let $\mathcal{X}$ and $\mathcal{Y}$ denote respectively the topological space
 $\H_{L}(O)$ with the topologies inherited by $C^\infty(O)$
 and $L^1_{\mathrm{loc}}(O)$.
 Then  $\mathcal{X}$ and $\mathcal{Y}$ are Fréchet spaces, since, if a sequence $u_n\in \H_{L}(O)$
 converges to $u$ uniformly on the compact sets of $\Omega$ or, more generally
 in $L^1_{\textrm{loc}}$,
 $$0=\int u_n\,L^*\varphi \xrightarrow{n\to\infty} \int u\, L^*\varphi,\qquad
 \forall\,\,\varphi \in C_0^\infty(O,\R).$$
 Now, the identity map $\iota:\mathcal{X}\to \mathcal{Y}$ is
 trivially linear, bijective and continuous, whence, by
 the Open Mapping Theorem, $\iota$ is a homeomorphism, whence the mentioned topologies coincide.\qed
\end{remark}
 We next resume our main proof. The set $\{g_x\}_{x\in\Omega}$ is
 bounded in $L^1(\Omega)$, since
  $$0\leq \int_\Omega g_x(y)\,\d y= G(1)(x)\leq \max_{\overline{\Omega}} G(1).$$
 A fortiori, the set  $\{g_x\}_{x\in\Omega}$ is also bounded
 in the topological vector space $L^1_{\mathrm{loc}}(\Omega)$.
 We next fix two disjoint open sets $U,W$
 with closures contained in $\Omega$.
 The family of the restrictions
 $$\Big\{(g_x)\big|_U\Big\}_{x\in W} $$
 is contained in the space of the $(\LL_\e)^*$-harmonic functions
 on $U$.
 By Remark \ref{rem.topolocinfyt},
 the set $\mathcal{G}$ is also bounded in the
 topological vector space 
 $$\mathcal{H}_{\displaystyle (\LL_\e)^*}(U),\quad \text{endowed with the $C^\infty$-topology.}$$
 This means that, for every
 compact set $K\subset U$ and for every $m\in \mathbb{N}$, there exists
 a constant $C(K,m)>0$ such that
\begin{equation}\label{staellatoplo}
    \sup_{|\alpha|\leq m}\,\,\sup_{y\in K}
 \bigg| \Big(\frac{\de}{\de y}\Big)^\alpha g(x,y) \bigg| \leq C(K,m),
\quad \text{uniformly for $x\in W$.}
\end{equation}
 Following Bony \cite[Section 6]{Bony}, we introduce the operator
 $F$ transforming any distribution $T$ compactly supported in $U$ into the function
 on $W$ defined by
 $$F(T):W\longrightarrow \R,\qquad F(T)(x):=\lan  T, g_x\ran\quad (x\in W).  $$
 The definition  is well-posed since $g_x\in C^\infty (U,\R)$ (and $T$ is compactly
 supported in $U$). We claim that $F(T)\in C^\infty(W,\R)$. Once this is proved,
 by the Schwartz Kernel Theorem (see e.g., \cite[Section 11]{Dieudonne} or
 \cite[Chapter 50]{Treves}),
 we can conclude that $g(x,y)$ is smooth on $W\times U$.
 By the arbitrariness of the disjoint open sets $U,W$ this proves that
 $g(x,y)$ is smooth out of the diagonal of $\Omega\times \Omega$, as desired.

 As for the proof of the claimed $F(T)\in C^\infty(W,\R)$, we can take (say, by some appropriate convolution)
  a sequence of continuous functions $f_n$, supported in $U$, converging
 to $T$ in the weak sense of distributions; due to the compactness of the supports (of the $f_n$ and of $T$),
 $$\lim_{n\to \infty}\int_U f_n\,\varphi=\lan T,\varphi\ran,\quad \text{for every $\varphi\in C^\infty(U,\R)$.} $$
 We are hence entitled to take $\varphi=g_x$ (for any fixed $x\in W$). From
 \eqref{quasimoltk} we get
\begin{equation}\label{numerata}
 \lim_{n\to \infty} G(f_n)(x)=\lan T, g_x\ran=F(T)(x),\quad \text{for any $x\in W$.}
\end{equation}
 We now prove that $F(T)\in L^\infty(W)$; this follows from the next calculation (here
 $C>0$ and $m\in \mathbb{N}$ are constants depending on $T$ and on the compact set $\overline{U}$)
\begin{align*}
    \|F(T)\|_{L^\infty}=\sup_{x\in W} |\lan T, g_x\ran|\leq
   \sup_{x\in W} C\sum_{|\alpha|\leq m}
  \sup_{y\in \overline{U}}\bigg| \Big(\frac{\de}{\de y}\Big)^\alpha g(x,y) \bigg|
\stackrel{\eqref{staellatoplo}}{\leq } \widetilde{C}(\overline{U},m)<\infty.
\end{align*}
 We finally prove that $\LL_\e (F(T))=0$ in the weak sense of distributions on $W$;
 by the hypoellipticity of $\LL_\e$ this will yield the smoothness of $F(T)$ on $W$.
 We aim to show that,
 $$\int_W F(T)(x)\,(\LL_\e)^*\varphi(x)\,\d x=0\qquad
\text{for any $\varphi\in C_0^\infty (W)$}. $$
 Now, the left-hand side is (by \eqref{numerata})
 $$\int \lim_{n\to \infty} G(f_n)(x)\,(\LL_\e)^*\varphi(x)\,\d x.$$
 If a dominated convergence can be applied, this is equal to
  $$\lim_{n\to \infty} \int_W  G(f_n)(x)\,(\LL_\e)^*\varphi(x)\,\d x
 \eqref{DIRIEQGreenBIS}{=}
 -\lim_{n\to \infty} \int_W f_n(x)\,\varphi(x)\,\d x=0,
$$
 the last equality descending from the fact that the $f_n$
 are supported in $U$ for every $n$.
 We are then left with showing that the dominated convergence
 is fulfilled: this is a consequence of
\eqref{staellatoplo}, of the boundedness of $F(T)$  on $W$, and of
 the fact that the convergence in \eqref{numerata} is indeed uniform w.r.t.\,$x\in W$
 (a general result of distribution theory: the uniform
  convergence for sequences of distributions on bounded sets).\medskip

 \textsc{Step IV.}
 We are finally ready to introduce our kernel
\begin{equation}\label{kkkkkkkkk}
  k:\Omega\times\Omega \longrightarrow [0,\infty),\qquad
 k(x,y):=\frac{g(x,y)}{V(y)}.
\end{equation}
 Clearly, from \eqref{quasimoltk} and \eqref{autoagg} we immediately have
\begin{gather}\label{quasimoltkkkkkk}
  G(f)(x)=\int_{\Omega} f(y)\,k(x,y)\,\d \nu(y),\quad \text{for
  every $f \in C(\overline{\Omega},\R)$ and every $x\in  \Omega$.}
\end{gather}
 This gives the representation \eqref{gprororfondam} whilst
 \eqref{gprororfondam2PROPR2}
 follows from \eqref{DIRIEQkernelBISTERQuater}.

 The integrability of $k(x,\cdot)$ in $\Omega$ is a consequence of
 $g_x\in L^1(\Omega)$ (and the positivity of the continuous function $V$ on $\RN$).
 Moreover, $k$ is smooth on $\Omega\times \Omega$ deprived of the diagonal by Step III.
 Also,  the nonnegative function $k$ is integrable on $\Omega\times \Omega$ as this computation shows:
 $$0\leq \int_{\Omega\times \Omega} k(x,y)\, \d x \d y=
 \int_{\Omega} \Big(\int_{\Omega} \frac{1}{V(y)}\,k(x,y)\,\d\nu(y)\Big) \d x
 \stackrel{\eqref{quasimoltkkkkkk}}{=} \int_{\Omega} G(1/V)(x)\,\d x<\infty, $$
 the last inequality following from the continuity of $G(1/V)$ on the compact set
 $\overline{\Omega}$.

 For fixed $x\in \Omega$, the $\LL_\e$-harmonicity of the function $k(x,\cdot)$
 in $\Omega\setminus\{x\}$ is a consequence of the following computation
 $$0\stackrel{\eqref{DIRIEQkernelBISTERQuater55}}{=}(\LL_\e)^* g_x
   \stackrel{\eqref{autoaggBISSS}}{=} V\,\LL_\e\Big(\frac{g_x}{V}\Big)
 \stackrel{\eqref{kkkkkkkkk}}{=}
 V\,\LL_\e(k(x,\cdot)).$$
 The fact that $V$ is positive then gives
 $\LL_\e(k(x,\cdot))=0$ in  $\Omega\setminus\{x\}$.
 From the SMP for $\LL_\e=\LL-\e$ in Remark \ref{PMFancheconc},
 we deduce that the nonnegative function
 $k(x,\cdot)$ (which is $\LL_\e$-harmonic in $\Omega\setminus\{x\}$)
 cannot attain the (minimal) value $0$; therefore
 $k(x,\cdot)>0$  on the connected open set $\Omega\setminus\{x\}$.

 A crucial step consists in proving the symmetry
 property \eqref{gprororfondam2}.
 We take any nonnegative $\varphi\in C_0^\infty(\Omega,\R)$ and
 we set (note the reverse order of $x$ and $y$, if compared to $G(\varphi)$)
 $$\Phi(x)=\int_\Omega \varphi(y)\,k(y,x)\,\d\nu(y),\qquad x\in\Omega. $$
 We claim that $\Phi\geq G(\varphi)$ on $\Omega$; once the claim is proved,
 from \eqref{quasimoltkkkkkk}
 we infer that
  $$\int_{\Omega} \varphi(y)\,k(x,y)\,\d \nu(y) \leq \int_\Omega \varphi(y)\,k(y,x)\,\d\nu(y),\qquad x\in\Omega. $$
 The arbitrariness of $\varphi$ will then give $k(x,y)\leq k(y,x)$
 (recalling that $\d\nu=V(y)\,\d y$ with positive $V$) for every $y\in\Omega$;
 since $x,y\in\Omega$ are arbitrary, we get $k(x,y)=k(y,x)$
 on $\Omega\times \Omega$. We prove the claim. We observe that
 $\Phi$ is continuous on $\Omega$ and that $\LL_\e\Phi=-\varphi$
 in $\mathcal{D}'(\Omega)$, as the following computation shows
 ($\psi\in C_0^\infty(\Omega,\R)$ is arbitrary):
\begin{align*}
    &\int_\Omega \Phi(x)\,(\LL_\e)^*\psi(x)\,\d x=
  \int_\Omega
\varphi(y)\,\Big(
     \int_\Omega
  k(y,x)\,
  \,(\LL_\e)^*\psi(x)\,\d x
  \Big)
 \d\nu(y)\\
 &\,\,\,\,=
  \int_\Omega
\varphi(y)\,\Big(
     \int_\Omega
  k(y,x)\,
  \,\frac{(\LL_\e)^*\psi(x)}{V(x)}\,\d \nu(x)
  \Big)
 \d\nu(y)\\
 &\stackrel{\eqref{autoaggBISSS}}{=}
   \int_\Omega
\varphi(y)\,\Big(
     \int_\Omega
  k(y,x)\,
  \,\LL_\e\Big(\frac{\psi(x)}{V(x)}\Big)\,\d \nu(x)
  \Big)
 \d\nu(y)\\
 &\stackrel{\eqref{gprororfondam2PROPR2}}{=}
  -  \int_\Omega
\varphi(y)\,\frac{\psi(y)}{V(y)}\,
 \d\nu(y)=
 -  \int_\Omega
 \varphi(y)\,\psi(y)\,\d y.
\end{align*}
 From the hypoellipticity of $\LL_\e$ we get
 $\Phi\in C^\infty(\Omega,\R)$ and
 $\LL_\e\Phi =-\varphi$ point-wise.
 We now apply the WMP in Remark \ref{PMFancheconc2} to
 the operator $\LL_\e=\LL-\e$
 and to the function
 $G(\varphi)-\Phi$: this function is smooth and $\LL_\e$-harmonic on $\Omega$,
 and $G(\varphi)-\Phi\leq G(\varphi)$ on $\Omega$ (since $\Phi$ is nonnegative), so that
 $$\limsup_{x\to x_0}(G(\varphi)-\Phi)(x)\leq \limsup_{x\to x_0} G(\varphi)(x)=0\quad
\text{for every $x_0\in\de\Omega$}.$$
 Therefore $G(\varphi)-\Phi\leq 0$ on $\Omega$ as claimed.

 We finally prove \eqref{gprororfondam2PROPR23}. Due to the symmetry property of $k$,
 \eqref{gprororfondam2PROPR23} will follow if we show that, given $x_0\in \Omega$
 and  $y_0\in\de\Omega$, one has
\begin{equation}\label{invertitag}
 \lim_{n\to \infty} k(y_n,x_0)=0,
\end{equation}
 for every sequence $y_n$ in $\Omega$ converging to $y_0$.
 To this end, we fix an open set $\Omega'$ containing $x_0$ and with closure contained in $\Omega$, and
 it is non-restrictive to suppose that $y_n\notin \Omega'$ for every $n$.
 The functions
 $$k_n:\Omega'\longrightarrow \R,\qquad k_n(x):=k(y_n,x),\quad x\in\Omega'$$
 are smooth and $\LL_\e$-harmonic in $\Omega'$.
 We also have $k_n\longto 0$ in $L^1(\Omega')$, as it follows from
\begin{align*}
    0&\leq \int_{\Omega'} k_n(x)\,\d x\leq \int_\Omega k(y_n,x)\,\d x=
  \int_\Omega \frac{g(y_n,x)}{V(x)}\,\d x\\
 & \leq \sup_{\Omega} \frac{1}{V}\,
 \int_\Omega g(y_n,x)\,\d x=
 \sup_{\Omega} \frac{1}{V}\, G(1)(y_n)\xrightarrow{n\to\infty} 0.
\end{align*}
 From Remark \ref{rem.topolocinfyt} we get that
 $k_n\longto 0$ in the Fréchet space $\H_{\LL_\e}(\Omega')$
 with the $C^\infty$-topology, so that
 $k_n\longto 0$ uniformly on the compact sets of $\Omega'$ and in particular
 point-wise on $\Omega'$.\qed
\section{The Harnack inequality}\label{sec:Harnackvera}
 We begin by proving the next crucial lemma. This is the first time that, broadly
 speaking, the PDOs $\LL$ and the perturbed $\LL-\e$ clearly interact.
\begin{lemma}\label{lem.crustimabasso}
 Let $\LL$ be as in \eqref{mainLL} and let it satisfy \emph{(NTD)} and \emph{(HY)$_\e$}.
 Let $\Omega$ be an open set in $\RN$ as in the thesis
 of Lemma \ref{th.localDiri}, and let $\Omega'$ be an open set containing
 $\overline{\Omega}$.
 Finally, we denote by $k_\e$ the Green kernel related to $\LL_\e$ and to the set $\Omega$
 (as in Theorem \ref{th.greeniani}).

 Then we have the estimate
\begin{equation}\label{crustimabasso.EQ1}
    u(x)\geq \e \int_{\Omega} u(y)\,k_\e(x,y)\,\d\nu(y),\quad \forall\,x\in\Omega,
\end{equation}
 holding true for every smooth nonnegative $\LL$-harmonic function $u$ in $\Omega'$.
\end{lemma}
\begin{proof}
 We consider the function $v(x)=\int_{\Omega} u(y)\,k_\e(x,y)\,\d\nu(y)$ on
 $\Omega$. From \eqref{gprororfondam} (and the definition of Green operator)
 we know that $v=G_\e(u)$, where $G_\e$ is the Green operator related to $\LL_\e$
 (and to the open set $\Omega$); moreover, since $u$ is smooth (by assumption) on $\overline{\Omega}$,
 we know from Lemma \ref{th.localDiri} (and the hypoellipticity of $\LL_\e$)
 that $v\in C^\infty(\Omega)\cap C(\overline{\Omega})$ is the solution of
 \begin{gather}\label{DIRIEQHARNACK}
    \left\{
      \begin{array}{ll}
        \LL_\e v=-u & \hbox{on $\Omega$}, \\
        v=0 & \hbox{on $\de\Omega$.}
      \end{array}
    \right.
 \end{gather}
 This gives $\LL_\e(\e\,v-u)=-\e\,u-(\LL-\e)u=-\e\,u+\e\,u=0$ on $\Omega$;
 moreover, on $\de\Omega$, $\e\,v-u=-u\leq 0$, by the nonnegativity of $u$.
 By the WMP in Remark \ref{PMFancheconc2}, we get $\e\,v-u\leq 0$
 on $\Omega$ which is equivalent to
 \eqref{crustimabasso.EQ1}.
\end{proof}
 We are ready for the proof of the Weak Harnack Inequality (for higher order
 derivatives).\medskip
\begin{proof}[Proof (of the Weak Harnack Inequality for derivatives, Theorem \ref{lem.crustimabassoTHEO})]
 We distinguish two ca\-ses: $y_0\notin K$ and $y_0\in K$.
 The second case can be reduced to the former. Indeed,
 let us assume we have already proved the theorem
 in the former case, and let $y_0\in K$.
 If we take any $y_0'\in O\setminus K$, and we
 consider the inequality
 $$  u(y_0') \leq C'\,u(y_0),$$
 resulting from
 \eqref{HarnackdeboleEQ1} by considering  $m=0$ and the compact set $\{y_0'\}$, we get
\begin{equation*}
 \sum_{|\alpha|\leq m}
 \sup_{x\in K} \Big| \frac{\de^\alpha u(x)}{\de x^\alpha}\Big|
 \stackrel{\eqref{HarnackdeboleEQ1}}{\leq} C\,u(y_0')
 \leq C\,C'\,u(y_0).
\end{equation*}
 We are therefore entitled to assume that $y_0\notin K$.
 By the aid of a classical argument (with a chain
 of suitable small open sets $\{\Omega_n\}_{n=1}^p$
 covering a connected compact set
 containing $K\cup\{y_0\}$), it is not restrictive to assume that $K\cup\{y_0\}\subset \Omega\subset
 \overline{\Omega}\subset O$, where $\Omega$ is
 one of the basis open sets constructed in Lemma
 \ref{th.localDiri}.

 Let $x_0\in K$ be arbitrarily fixed.
 The function $k_\e(x_0,\cdot)$  (the Green kernel related to $\LL_\e$ and $\Omega$)
 is \emph{strictly positive} in $\Omega\setminus\{x_0\}$ (this is a consequence
 of the SMP applied to the $\LL_\e$-harmonic function $k_\e(x_0,\cdot)$; see
 Theorem \ref{th.greeniani}). In particular,  since $y_0\notin K$, we infer that
 $k_\e(x_0,y_0)>0$. Hence, there exist a neighborhood $W$ of $x_0$
 (contained in $\Omega$)
 and a constant $\mathbf{c}=\mathbf{c}(\e,y_0,x_0)>0$ such that
\begin{equation}\label{HarnackdeboleEQ2}
 \inf_{z\in W} k_\e(z,y_0)\geq \mathbf{c}>0.
\end{equation}
 Our assumptions allow us to apply
 Lemma \ref{lem.crustimabasso}: hence, for every nonnegative
 $u\in \mathcal{H}_\LL(O)$, we have the following chain of inequalities
\begin{align*}
   u(y_0) &\stackrel{\eqref{crustimabasso.EQ1}}{\geq}
  \e \int_{\Omega}    u(z)\,k_\e(y_0,z)\,\d\nu(z)
 \geq
 \e \int_{W} u(z)\,k_\e(y_0,z)\,\d\nu(z)\\
 &\stackrel{\eqref{gprororfondam2}}{=}
 \e \int_{W} u(z)\,k_\e(z,y_0)\,\d\nu(z)
 \stackrel{\eqref{HarnackdeboleEQ2}}{\geq}
 \e\,\mathbf{c} \int_{W} u(z)\,\d\nu(z)\geq
 \e\,\mathbf{c}\,\inf_W V\, \int_{W} u(z)\,\d z.
\end{align*}
 Summing up, for every $x_0\in K$ there exist
 a neighborhood $W$ of $x_0$ and
 a constant $\mathbf{c}_1>0$ (also depending on $x_0$
 but independent of $u$) such that
\begin{equation}\label{HarnackdeboleEQ3}
    u(y_0) \geq
  \mathbf{c}_1\int_{W} u(z)\,\d z,
\end{equation}
 for every nonnegative $u\in \mathcal{H}_\LL(O)$.

 Next, from
 Remark \ref{rem.topolocinfyt}, we know that the
 hypothesis (HY) for $\LL$ ensures the equality of the topologies
 on $\H_{\LL}(W)$ inherited by the Fréchet spaces $C^\infty(W)$
 and $L^1_{\mathrm{loc}}(W)$.
 In particular, to any chosen open neighborhood
 $U$ of $x_0$ (with $\overline{U}\subset W$) we are given
 a positive constant $\mathbf{c}_2=\mathbf{c}_2(U,W,m)$ such that
  \begin{equation}\label{HarnackdeboleEQ4}
 \sum_{|\alpha|\leq m}
 \sup_{x\in U} \Big| \frac{\de^\alpha u(x)}{\de x^\alpha}\Big|
 \leq \mathbf{c}_2\,\int_{W} u(z)\,\d z,
\end{equation}
  for every nonnegative $u\in \mathcal{H}_\LL(O)$.
 Gathering together \eqref{HarnackdeboleEQ3}
 and \eqref{HarnackdeboleEQ4}, we infer that,
 for every $x_0\in K$ there exist
 a neighborhood $U$ of $x_0$ and
 a constant $\mathbf{c}_3>0$ (again depending on $x_0$
 but independent of $u$) such that
\begin{equation*}
    u(y_0) \geq
  \mathbf{c}_3
 \sum_{|\alpha|\leq m}
 \sup_{x\in U} \Big| \frac{\de^\alpha u(x)}{\de x^\alpha}\Big|,
\end{equation*}
 for every nonnegative $u\in \mathcal{H}_\LL(O)$.
 The compactness of $K$ allows us to derive
 \eqref{HarnackdeboleEQ1} from the latter inequality, and a covering argument.
\end{proof}

 We now present a proof of Theorem \ref{th.mokobre}, crucially based on \cite[Chapter I]{Brelot}.
\begin{proof}[Proof (of Theorem \ref{th.mokobre})]
 As anticipated in the Introduction, the proof is based in an essential way on the ideas
 by Mokobodzki-Brelot in \cite[Chapter I]{Brelot}, ensuring the equivalence
 of the Strong Harnack Inequality with a series of properties comprising
 the Weak Harnack Inequality, provided some assumptions are fulfilled.
 We furnish some details in order to
 be oriented through these equivalent properties.

 We denote by $\mathcal{H}_L$ the harmonic sheaf on $\RN$ defined by
 $O\mapsto \mathcal{H}_L(O)$ (here $O\subseteq\RN$ is any open set).
 Under the assumptions of (Regularity) and (Weak Harnack Inequality),
 Brelot proves that (see \cite[pp.22--24]{Brelot}),
 for any connected open set $O\subseteq\RN$, and any
 $x_0\in O$, the set
 \begin{equation}\label{fix0}
    \Phi_{x_0}:=\Big\{h\in \mathcal{H}_L(O)\,:\,h\geq 0,\quad h(x_0)=1\Big\}
 \end{equation}
 is equicontinuous at $x_0$. The proof of this fact rests on
 some results of Functional Analysis related to the family
 of the so-called harmonic measures $\{\mu^\Omega_x\}_{x\in \de\Omega}$
 associated with $L$
 (and on basic properties of the harmonic sheaf $\mathcal{H}_L$).
 Next, we show how to prove \eqref{SHIvera} starting from
 the equicontinuity of $\Phi_{x_0}$ at $x_0$. Indeed,
 let $K\subset O$, where $K$ is compact and $O$ is an open and connected subset
 of $\RN$. By possibly enlarging $K$, we can suppose that $K$ is connected
 as well. Let $u\in \mathcal{H}_L(O)$ be nonnegative.
 If $u\equiv 0$ then \eqref{SHIvera} is trivial;
 if $u$ is not identically zero then
 (from the Weak Harnack Inequality) one has
 $u>0$ on $O$.
 For every $x\in K$, the equicontinuity of $\Phi_x$ ensures the existence
 of $\delta(x)>0$ such that (with the choice $h=u/u(x)$ in \eqref{fix0})
\begin{equation}\label{fix0EQ1}
 \frac{1}{2}\,u(x)\leq u(\xi)\leq \frac{3}{2}\,u(x),\quad
 \text{for all $\xi\in B_x:=B(x,\delta(x))$.}
\end{equation}
 From the open cover $\{B_x\}_{x\in K}$ we can extract a finite
 subcover $B_{x_1},\ldots,B_{x_p}$ of $K$. It is also non-restrictive
 (since $K$ is connected)
 to assume that the elements of this subcover are chosen in such a way that
 $$B_{x_1}\cap B_{x_2}\neq \varnothing,\quad
 (B_{x_1}\cup B_{x_2})\cap B_{x_3}\neq \varnothing,\quad\ldots\quad
 (B_{x_1}\cup\cdots \cup B_{x_{p-1}})\cap B_{x_p}\neq \varnothing.   $$
 From \eqref{fix0EQ1} it follows
 \eqref{SHIvera} with $K$ replaced by $B_{x_1}$
 (with $M=3$); since $B_{x_1}$ intersects $B_{x_2}$, one can use
 again \eqref{fix0EQ1} in order to prove
 \eqref{SHIvera} with $K$ replaced by $B_{x_1}\cup B_{x_2}$
 (with $M=3^2$); by proceeding in an inductive way, one can prove
 \eqref{SHIvera} with $K$ replaced by $B_{x_1}\cup\cdots\cup B_{x_p}$ (and $M=3^p$), and this finally proves
 \eqref{SHIvera}, since $B_{x_1}\cup\cdots\cup B_{x_p}$ covers $K$.
\end{proof}
\begin{remark}\label{rem.precisiamo}
 Following Brelot \cite[pp.14--17]{Brelot},
 it being understood that axiom (Regularity) in
 Theorem \ref{th.mokobre} holds true,
 the axiom (Weak Harnack Inequality) can be replaced by any of the following
 equivalent assumptions (see also Constantinescu and Cornea \cite{ConstantinescuCornea}):
\begin{description}
  \item[(Brelot Axiom)]
  For every connected open set $O\subseteq \RN$,
  if $\mathcal{F}$ is an up-directed\footnote{$\mathcal{F}$
 is said to be up-directed if for any $u,v\in \mathcal{F}$ there exists $w\in \mathcal{F}$
 such that $\max\{u,v\}\leq w$.}
 family of $L$-harmonic functions in $O$,
 then $\sup\limits_{u\in \mathcal{F}}u$ is either $+\infty$
 or it is $L$-harmonic in $O$.

  \item[(Harnack Principle)]
   For every connected open set $O\subseteq \RN$,
  if $\{u_n\}_n$ is a non-de\-crea\-sing sequence
 of $L$-harmonic functions in $O$,
 then $\lim\limits_{n\to \infty} u_n$ is either $+\infty$
 or it is an $L$-harmonic function in $O$.
\end{description}
\end{remark}
 \noindent We are ready to derive our main result for this section: due to all our preliminary results,
 the proof is now a few lines argument.
\begin{proof}[Proof (of Harnack Inequality, Theorem \ref{lem.crustimabassoTHEOFORTE})]
 Due to Theorem \ref{th.mokobre},
 it suffices to prove that our operator $\LL$ as in the statement of Theorem \ref{lem.crustimabassoTHEOFORTE}
 satisfies the properties named (Regularity) and (Weak Harnack Inequality)
 in Theorem \ref{th.mokobre}:
 the former is a consequence of Lemma \ref{th.localDiri}
 (with $f=0$), whilst the latter
 follows from Theorem \ref{lem.crustimabassoTHEO}.
\end{proof}

 \section{Appendix: The Dirichlet problem for $\LL$}\label{sec:Dirichlet}
 The aim of this appendix is to prove Lemma \ref{th.localDiri}
 under the following more general form in Theorem \ref{th.localDiriMIGLIO}:
 our slightly more general framework (we indeed deal with general
 hypoelliptic operators which are non-totally degenerate at every point)
 compared to the one considered by Bony in \cite{Bony} (where
 H\"ormander operators are concerned) does not present much more difficulties than the one in \cite[Section 5]{Bony},
 and the proof is given for the sake of completeness only.
\begin{theorem}\label{th.localDiriMIGLIO}
 Suppose that $L$ is an operator on $\RN$ of the form
\begin{equation}\label{mainLLMIGLIO}
 L=\sum_{i,j=1}^N\alpha_{i,j}\frac{\de^2}{\de x_i\de x_j}+
  \sum_{i=1}^N \beta_i\frac{\de}{\de x_i}+\gamma,
\end{equation}
 with  $\alpha_{i,j},\beta_i,\gamma\in C^\infty (\RN,\R)$,
 with $(\alpha_{i,j})$ symmetric and
 positive semi-definite. We assume that $L$ is non-totally
 degenerate at every $x\in\RN$ and that $L$ is $C^\infty$-hypoelliptic in every open set.\medskip

 Then there exists a basis for the Euclidean topology of $\RN$ made of open sets
 $\Omega$ with the following properties:
 for every continuous function $f$ on $\overline{\Omega}$ and for every continuous
 function $\varphi$ on $\de\Omega$, there exists one and only one solution
 $u\in C(\overline{\Omega},\R)$ of the Dirichlet problem
 \begin{gather}\label{DIRIEQMIGLIO}
    \left\{
      \begin{array}{ll}
        L u=-f & \hbox{on $\Omega$  (in the weak sense of distributions),} \\
        u=\varphi & \hbox{on $\de\Omega$ (point-wise).}
      \end{array}
    \right.
 \end{gather}
 Furthermore, if $f,\varphi\geq 0$ then $u\geq 0$ as well. Finally, if $f$ belongs to $C^\infty(\Omega,\R)
 \cap C(\overline{\Omega},\R)$, then the same is true of $u$, and $u$
is a classical solution of \eqref{DIRIEQMIGLIO}.

 Finally, if the zero-order term $\gamma$ of $L$ is non-positive on $\R$, the
 above basis $\{\Omega\}$ does not depend on $\gamma$. If $\gamma<0$,
 the basis  $\{\Omega\}$ only depends on the principal matrix $(\alpha_{i,j})$
 of $L$.
\end{theorem}
 The key step is to construct a basis for the Euclidean topology of $\RN$ as follows:
\begin{lemma}\label{lem.regol}
 Let $A(x)=(a_{i,j}(x))$ be a matrix with real-valued continuous entries on $\RN$,
 which is symmetric, positive semi-definite and non-vanishing at a point $x_0\in\RN$.

 Then, there exists a basis of connected open neighborhoods $\mathcal{B}_{x_0}$ of $x_0$
 such that any $\Omega\in \mathcal{B}_{x_0}$ satisfies the following property:
  for every $y\in \de\Omega$ there exists $\nu\in \RN\setminus\{0\}$
   such that $\overline{B(y+\nu,|\nu|)}$ intersects $\overline{\Omega}$ at $y$ only, and such that
  \begin{equation}\label{normaleest}
    \lan A(y)\,\nu , \nu \ran >0.
  \end{equation}
\end{lemma}
\begin{proof}
 By the assumptions on $A(x_0)$ there exists a unit vector $h_0$
 such that
  \begin{equation}\label{normaleestEQ1}
 \lan A(x_0) h_0, h_0\ran >0.
  \end{equation}
 Following the idea of Bony \cite{Bony},
 we choose the neighborhood basis $\mathcal{B}_{x_0}=\{\Omega(\e)\}$ as follows:
 $$\Omega(\e):=B(x_0+\e^{-1}\,h_0,\e^{-1}+\e^2)\cap B(x_0 - \e^{-1}\,h_0, \e^{-1}+\e^2) .$$
 It suffices to show that there exists $\overline{\e}>0$ such that
 every  $\Omega(\e)$ with $0<\e\leq \overline{\e}$
 satisfies the requirement of the lemma. Now,
 the set $\Omega(\e)$ (which is
 trivially an open neighborhood of $x_0$) shrinks to $\{x_0\}$
 as $\e$ shrinks to $0$. Moreover, every
 $y\in \de \Omega(\e)$ belongs to one at least of the
 spheres $\de B(x_0\pm\e^{-1}\,h_0,\e^{-1}+\e^2)$; accordingly,
 we choose
 $$\nu=\nu_\e(y):=\frac{y-(x_0\pm\e^{-1}\,h_0)}{\e^{-1}+\e^2}$$
 to get the geometric condition $\overline{B(y+\nu,|\nu|)}\cap \overline{\Omega(\e)}=\{y\}$.
 It obviously holds that $\nu_\e(y)$ tends to $h(x_0)$ as $\e\to 0$
 (uniformly for bounded $x_0,y,h_0$), so that
 \eqref{normaleest} follows from \eqref{normaleestEQ1}
 by continuity arguments, for any $0\leq \e\leq \overline{\e}$, with $\overline{\e}$
 conveniently small.
\end{proof}
 We proceed with the proof of
 Theorem \ref{th.localDiriMIGLIO} by constructing, for any given $x_0\in\RN$,
 a basis of neighborhoods of $x_0$ as required.
 The crucial step is to reduce $L$ to some equivalent operator $\widetilde{L}$
 with zero-order term $\widetilde{L}(1)$ which is strictly negative around $x_0$.
 We observe that this procedure is not necessary if $\gamma=L(1)$ is already known to be negative on $\RN$.
 In general, we let
  $$\widetilde{L} u:=w\,L(w\,u),\quad \text{where $w(x)=1-M\,|x-x_0|^2$},$$
 with $M\gg 1$ to be chosen.
 Let us denote by $B(x_0)$ the Euclidean
 ball of centre $x_0$ and radius $1/\sqrt M$.
 It is readily seen that the second order parts of $L$ and $\widetilde{L}$
 are equal, modulo the factor $w^2$.
 This shows that $\widetilde{L}$ is non-totally degenerate at any point of
 $B(x_0)$ and that the principal matrix of $\widetilde{L}$ is
 symmetric and positive semi-definite at any point of $B(x_0)$.
 Since
\begin{align*}
 \widetilde{L}(1)(x)&=w^2(x)\,\gamma(x)-2M\textrm{trace}(A(x))-2M\sum_{i=1}^N \beta_i(x)\,(x-x_0)_i,
\end{align*}
 if we choose $M$ so large that $M>\gamma(x_0)/(2\,\textrm{trace}(A(x_0)))$
 (we recall that $\textrm{trace}(A(x))>0$ at any $x$ since $L$ is non-totally degenerate
 at any point), then $\widetilde{L}(1)(x_0)<0$. By continuity,
 there exists $r>0$ small enough such that $B'(x_0):=B(x_0,r)\subseteq B(x_0)$
 and such that  $\widetilde{L}(1)<0$ on the closure of $B'(x_0)$.
 We explicitly remark (and this will prove the final statement of the theorem)
 that the condition $\gamma\leq 0$ allows us to take $M=1$ for all $x_0$
 and to use the bound
\begin{align*}
 \widetilde{L}(1)(x)&\leq -2\textrm{trace}(A(x))-2\sum_{i=1}^N \beta_i(x)\,(x-x_0)_i,
\end{align*}
 in order to chose $r$ independently of $\gamma$.
\begin{remark}\label{rem.WMPsullabase}
 Classical arguments, \cite{lanco_maxprinc}, show that, due to the
 strict negativity of $\widetilde{L}(1)$ on $B'(x_0)$, the operator $\widetilde{L}$
 satisfies the Weak Maximum Principle on every open subset of $B'(x_0)$, that is:
\begin{equation}\label{WMPB'}
 \left\{
   \begin{array}{ll}
     \Omega\subset B'(x_0),\,\,u\in C^2(\Omega,\R)\\
     \widetilde{L} u\geq 0\,\,\text{on $\Omega$} \\
     \limsup\limits_{x\to y} u(x)\leq 0\,\,\text{for every $y\in\de\Omega$}
   \end{array}
 \right.
 \qquad\Longrightarrow\qquad
 u\leq 0\,\,\text{on $\Omega$.}
\end{equation}
\end{remark}
 The rest of the proof consists
 in demonstrating the following statement:
\begin{description}
  \item[(S)]
 \emph{there exists a basis ${\mathcal{B}}_{x_0}$ of neighborhoods $\Omega$ of $x_0$ all contained in $B'(x_0)$
 with the properties required in
 Theorem \ref{th.localDiriMIGLIO} relative to $\widetilde{L}$ (in place of $L$)}.
\end{description}
 Once this is proved,
 given any  $\Omega\in {\mathcal{B}}_{x_0}$, any $f\in C(\overline{\Omega},\R)$ and any
 $\varphi\in C(\de\Omega,\R)$, we obtain the solution $\widetilde{u}$ of the problem
 \begin{gather}\label{problema.tildato}
    \left\{
      \begin{array}{ll}
        \widetilde{L} \widetilde{u}=-w\,f & \hbox{on $\Omega$  (in the weak sense of distributions),} \\
        \widetilde{u}=\varphi/w & \hbox{on $\de\Omega$ (point-wise);}
      \end{array}
    \right.
 \end{gather}
 then we set $u:=w\,\widetilde{u}$, and a simple verification shows that
 $u$ solves \eqref{DIRIEQMIGLIO}, so that existence is proved.
 As for uniqueness, it suffices to observe that for any fixed $\Omega\in{\mathcal{B}}_{x_0}$,
 to any solution $u$ of
 \eqref{DIRIEQMIGLIO} on $\Omega$, there corresponds a solution
 $\widetilde{u}=u/w$ of \eqref{problema.tildato} (which is unique, as it is claimed in (S)).
 Finally all the other requirements
 on $u$ in the statement of Theorem \ref{th.localDiriMIGLIO}
 are satisfied, since $w$ is positive and smooth on $\Omega\subseteq B(x_0)$.
\begin{remark}\label{rem-ipoLP}
 \emph{We remark that the operator $\widetilde{L}$ is $C^\infty$-hypoelliptic on every open subset of
 $B(x_0)$.}\medskip

 Indeed, for any open sets $V,V'$ such that $V\subseteq V'\subseteq B(x_0)$,
 a distribution $u\in \mathcal{D}'(V')$
 such that $\widetilde{L}u =f\in C^\infty(V,\R)$ satisfies
 $L(w\,u)=f/w\in C^\infty(V,\R)$; thus, by the hypoellipticity of
  $L$, we infer that $w\,u\in C^\infty(V,\R)$ so that $u\in C^\infty(V,\R)$
 (recalling that $w\neq 0$ on $B(x_0)$).
\end{remark}
 We are then left to prove statement (S).
 From now on we choose a neighborhood basis $\mathcal{B}_{x_0}$ of $x_0$ consisting of open sets
 (contained in $B'(x_0)$)
 as in Lemma \ref{lem.regol} relative to the principal matrix $\widetilde{A}$ of the operator $\widetilde{L}$
 (the matrix $\widetilde{A}(x_0)$ is symmetric, positive semi-definite and non vanishing, as already discussed).
 We will show that any $\Omega\in \mathcal{B}_{x_0}$ has the requirements in statement (S).
 For the uniqueness part, it suffices to use in a standard way
 the WMP in Remark \ref{rem.WMPsullabase} jointly with the
 hypoellipticity condition in Remark \ref{rem-ipoLP}.
 As for existence, we split the proof in several steps and, to simplify the notation, we write
 $P$ instead of $\widetilde{L}$.\medskip

 (I): \emph{$f$ smooth and $\varphi\equiv 0$}.
 We fix $\Omega$ as above, $f\in C^\infty(\Omega,\R)\cap C(\overline{\Omega},\R)$
 and $\varphi\equiv 0$. We use a standard elliptic approximation argument.
 For every $n\in \mathbb{N}$ we set
$$P_n:=P+\frac{1}{n}\sum_{j=1}^N \Big(\frac{\de}
 {\de x_j}\Big)^2.$$ We observe that:
\begin{itemize}
  \item[-] $P_n$ is uniformly elliptic on $\RN$;
  \item[-] the zero-order term $P_n(1)=P(1)\,\,(=\widetilde{L}(1))$ is (strictly) negative on $\Omega$;
  \item[-] $\Omega$ satisfies an exterior ball condition, due to Lemma \ref{lem.regol};
  \item[-] $f\in C^\infty(\Omega,\R)$.
\end{itemize}
 These conditions imply the existence (see e.g., Gilbarg and Trudinger \cite{GilbargTrudinger}) of a classical solution $u_n\in C^\infty(\Omega,\R)\cap C(\overline{\Omega},\R)$
 of the Dirichlet problem
\begin{gather*}
    \left\{
      \begin{array}{ll}
        P_n u_n=-f & \hbox{on $\Omega$} \\
        u_n=0 & \hbox{on $\de\Omega$.}
      \end{array}
    \right.
 \end{gather*}
  Let $c_0>0$ be such that $P(1)<-c_0$ on the closure of $B'(x_0)$.
  With this choice, we observe that (setting $\|f\|_\infty=\sup_{\overline{\Omega}}|f|$)
 $$
 \left\{
   \begin{array}{ll}
    P_n \Big(\pm u_n-\dfrac{\|f\|_\infty}{c_0}\Big)=\mp f-\dfrac{\|f\|_\infty}{c_0}\,P(1)\geq
 \mp f+  \dfrac{\|f\|_\infty}{c_0}\,c_0\geq 0 & \hbox{on $\Omega$} \\
    \pm u_n-\dfrac{\|f\|_\infty}{c_0}=-\dfrac{\|f\|_\infty}{c_0}\leq 0 & \hbox{on $\de\Omega$.}
   \end{array}
 \right.$$
 Arguing as in Remark \ref{rem.WMPsullabase}, the Weak Maximum Principle for $P_n$
 proves that
\begin{equation}\label{stimaunnnn}
   \|u_n\|_\infty= \sup_{x\in \overline{\Omega}} |u_n(x)|\leq  \dfrac{\|f\|_\infty}{c_0}\quad
\text{uniformly for every $n\in \mathbb{N}$}.
\end{equation}
 This provides us with a subsequence of $u_n$ (still denoted by $u_n$)
 and a function $u\in L^\infty(\Omega)$ such that $u_n$ tends to $u$
 in the weak$^*$ topology, that is
\begin{equation}\label{stimaunnnnEQ1}
 \lim_{n\to \infty} \int_\Omega u_n\,h= \int_\Omega u\,h,\quad
\text{for all $h\in L^1(\Omega)$.}
\end{equation}
 Moreover one knows that
\begin{equation}\label{stimaunnnnEQ1bis}
  \|u\|_{L^\infty(U) }\leq \limsup_{n\to \infty} \|u_n\|_{L^\infty(U) },\quad
\text{for all $U\subseteq \Omega$.}
\end{equation}
 From \eqref{stimaunnnnEQ1} it easily follows that
\begin{equation*}
 \int_\Omega u\,P^*\psi= -\int_\Omega f\,\psi,\quad
\text{for all $\psi\in C_0^\infty(\Omega,\R)$.}
\end{equation*}
  This means that $Pu=-f$ in the weak sense of distributions.
 As $P$ is hypoelliptic on every open set (Remark \ref{rem-ipoLP}),
 we infer that $u$ can be modified on a null set in such a way that
 $u\in C^\infty(\Omega,\R)$. Thus $P u=-f$ in the classical sense on $\Omega$.
 We aim to prove that $u$ can be continuously prolonged to $0$ on $\de\Omega$.
 To this end, given any $y\in \de\Omega$, in view of
 Lemma \ref{lem.regol} (and the choice of $\Omega$),
 there exists $\nu\in \RN\setminus\{0\}$
 such that $\overline{B(y+\nu,|\nu|)}$ intersects $\overline{\Omega}$ at $y$ only,
 and such that (see \eqref{normaleest})
  \begin{equation}\label{normaleestBBO}
    \lan \widetilde{A}(y)\,\nu , \nu \ran >0.
  \end{equation}
 As in the Hopf-type Lemma \ref{lem_Hopf}, we consider the function
 $$w(x) := e^{-\lambda|x - (y+\nu)|^2} - e^{-\lambda|\nu|^2},$$
 where $\lambda$ is a positive real number chosen in a moment.
 For every $n$ and for every $x$ one has
\begin{gather}\label{stimaintornidel}
\begin{split}
 P_n w(x) &=  Pw(x)+\frac{1}{n}\,e^{-\lambda|x - (y+\nu)|^2}\Big(4\lambda^2|x-(y+\nu)|^2
 -2\lambda N\Big)\\
 &\geq
 Pw(x)-2\lambda N e^{-\lambda|x - (y+\nu)|^2}.
\end{split}
\end{gather}
 If we set $P=\sum_{i,j}\widetilde{a}_{i,j}\de_{i,j}+\sum_j \widetilde{b}_j\de_j+\widetilde{c}$,
 a simple computation (similar to \eqref{lem_Hopf.EQ1ccc}) shows that
\begin{align*}
 &\Big( Pw(x)-2\lambda N e^{-\lambda|x - (y+\nu)|^2}\Big)\Big|_{x=y}\\
 &=
 e^{-\lambda|\nu|^2}\bigg(4\lambda^2
 \langle\widetilde{A}(y)\nu,\nu\rangle
 - 2\lambda\sum_{j = 1}^N\big(\widetilde{a}_{j,j}(y) - \widetilde{b}_j(y)\nu_j\big)
 -2\,\lambda\,N\bigg).
 \end{align*}
 Thanks to \eqref{normaleestBBO}, there exists $\lambda \gg 1$
 such that the above right-hand side is strictly  positive.
 Therefore, due to \eqref{stimaintornidel}
 there exist $\e>0$ and an open ball $V=B(y,\delta)$ (with $\e$ and $\delta$ \emph{independent of} $n$)
 such that
\begin{equation}\label{doppiastimadel}
 P_nw(x)\geq \e\quad \text{for every $x\in V$ and every $n\in\mathbb{N}$.}
\end{equation}
 We are willing to apply the Weak Maximum Principle for the operator $P_n$
 on the open set $\Omega\cap V$, and for the functions $M\,w\pm u_n$, where
 $M\gg 1$ is chosen as follows. First we have
 $$P_n (M\,w\pm u_n)=M\,P_n w\pm P_n u_n=M\,P_n w \mp f\geq M\,\e \mp f\geq M\,\e-\|f\|_\infty,\quad
 \text{in $\Omega\cap V$}. $$
 Consequently we first chose $M>\|f\|_\infty/\e$. Then we study the behavior of  $M\,w\pm u_n$
 on $$\de(\Omega\cap V)=[V\cap\de \Omega] \cup [\overline{\Omega}\cap \de V]=:\Gamma_1\cup \Gamma_2.$$
 Firstly, on $\Gamma_1$ we have $M\,w\pm u_n=M\,w\leq 0$ since $\Gamma_1\subseteq \RN\setminus B(y+\nu,|\nu|)$.
 Secondly, on $\Gamma_2$,
 $$M\,w\pm u_n \leq M\,\max_{\Gamma_2}w+\|u_n\|_\infty
 \stackrel{\eqref{stimaunnnn}}{\leq }
 M\,\max_{\Gamma_2}w+\dfrac{\|f\|_\infty}{c_0}.$$
 Since $\Gamma_2$ is a compact set on which $w$ is strictly negative, we have
 $\max_{\Gamma_2}w<0$ and the further choice
 $M\geq - \|f\|_\infty/(c_0\max_{\Gamma_2}w)$ yields
 $M\,w\pm u_n\leq 0$ on $\Gamma_2$. Summing up,
 $$\left\{
   \begin{array}{ll}
    P_n (M\,w\pm u_n)\geq 0 & \hbox{on $\Omega\cap V$} \\
    M\,w\pm u_n \leq 0 & \hbox{on $\de(\Omega\cap V)$.}
   \end{array}
 \right.$$
   The Weak Maximum Principle  yields $M\,w\pm u_n\leq 0$ on $\Omega\cap V$, that is
 (since $w<0$ on $\Omega$)
 $$ |u_n(x)|\leq M\,|w(x)|\quad \text {for every $x\in \Omega\cap V$ and for every $n\in \mathbb{N}$.}$$
 Since $w(y)=0$, for every $\sigma>0$ there exists an open neighborhood $W\subset V$ of $y$
 such that $\|w\|_{L^\infty(W)}<\sigma$; the above inequality then gives
 $\|u_n\|_{L^\infty(W\cap \Omega)}\leq M\,\sigma$. Jointly with
 \eqref{stimaunnnnEQ1bis} we deduce that
  $\|u\|_{L^\infty(W\cap \Omega)}\leq M\,\sigma$, so that $\lim_{\Omega\ni x\to y}u(x)=0$.
 From the arbitrariness of $y$, we obtain that $u$ prolongs to be $0$ on $\de\Omega$ with
 continuity.

 In order to complete the proof of (S), we are left to show that
 if $f\in C^\infty(\Omega,\R)\cap C(\overline{\Omega},\R)$
 is nonnegative, then the unique solution $u\in C(\overline{\Omega},\R)$ of
\begin{gather*}
    \left\{
      \begin{array}{ll}
        P u=-f & \hbox{on $\Omega$  (in the weak sense of distributions)} \\
        u=0 & \hbox{on $\de\Omega$ (point-wise)}
      \end{array}
    \right.
 \end{gather*}
 is nonnegative as well. From the hypoellipticity of $P$
 (see Remark \ref{rem-ipoLP}), we already know that
 $u\in C^\infty(\Omega,\R)$, and we can apply the WMP to $-u$
 (see Remark \ref{rem.WMPsullabase})
 to get $-u\leq 0$.\medskip

 (II): \emph{$f$ and $\varphi$ smooth}.
  We fix $\Omega$ as above, and $f$ is in $C^\infty(\Omega,\R)\cap C(\overline{\Omega},\R)$
  and $\varphi$ is the restriction to $\de\Omega$ of some function
 $\Phi$ which is smooth and defined on an open neighborhood of $\overline{\Omega}$.
 As in Step (I), we consider the unique solution $v\in C^\infty(\Omega,\R)\cap C(\overline{\Omega},\R)$ of
\begin{gather*}
    \left\{
      \begin{array}{ll}
        P v=-f-P\Phi & \hbox{on $\Omega$} \\
        v=0 & \hbox{on $\de\Omega$,}
      \end{array}
    \right.
 \end{gather*}
 and we observe that $u=v+\Phi$ is the (unique) classical solution of
\begin{gather*}
    \left\{
      \begin{array}{ll}
        P u=-f & \hbox{on $\Omega$} \\
        u=\Phi|_{\de\Omega}=\varphi & \hbox{on $\de\Omega$.}
      \end{array}
    \right.
 \end{gather*}
 If furthermore $f,\varphi\geq 0$, the nonnegativity of $u$ is a consequence of the
 WMP as in Step (I).\medskip

 (III): \emph{$f$ and $\varphi$ continuous}.
 Finally we consider
 $f\in C(\overline{\Omega},\R)$ and
 $\varphi\in C(\de\Omega,\R)$.
 By the Stone-Weierstrass Theorem, there
 exist polynomial functions $f_n,\varphi_n$
 uniformly converging to $f,\varphi$ respectively on $\overline{\Omega},\de\Omega$
as $n\to\infty$. As in Step (II), for every $n\in\mathbb{N}$ we consider the
 unique classical solution $u_n$ of
  \begin{gather*}
    \left\{
      \begin{array}{ll}
        P u_n=-f_n & \hbox{on $\Omega$} \\
        u_n=\varphi_n & \hbox{on $\de\Omega$.}
      \end{array}
    \right.
 \end{gather*}
 From the fact that
 $-c_0:=\max_{\overline{\Omega}} P(1)<0$, we
 can argue as in Step (I), obtaining the estimate
 $$\|u_n-u_m\|_{C(\overline{\Omega})}
 \leq \max\bigg\{\frac{1}{c_0}\,\|f_n-f_m\|_{C(\overline{\Omega})},
 \|\varphi_n-\varphi_m\|_{C({\de\Omega})}\bigg\}. $$
 This proves that there exists
 the uniform limit $u:=\lim_{n\to\infty} u_n$ in $C(\overline{\Omega},\R)$.
 Clearly one has: $u=\varphi$ point-wise on $\de\Omega$ and
 $Pu=-f$ in the weak sense of distributions on $\Omega$.
 From the hypoellipticity of $P$ (Remark \ref{rem-ipoLP})
 we infer that $f$ smooth implies $u$ smooth. Finally,
 suppose that $f,\varphi\geq 0$.
 By the Tietze Extension Theorem, we prolong $f$ out of $\overline{\Omega}$
 to a \emph{continuous} function $F$ on $\RN$;
 we consider a mollifying sequence $F_n\in C^\infty(\RN,\R)$
 uniformly converging to $F$ on the compact sets of $\RN$.
 Since mollification preserves the sign, the fact that
 $F|_{\overline{\Omega}}\equiv f\geq 0$ on $\overline{\Omega}$ gives that
 $F_n\geq 0$ on $\overline{\Omega}$.
 As above in this Step, we solve the problem
  \begin{gather*}
  \left\{
      \begin{array}{ll}
        P U_n=-F_n & \hbox{on $\Omega$} \\
        U_n=\varphi & \hbox{on $\de\Omega$,}
      \end{array}
    \right.
\qquad \text{with}\quad   U_n\in C^\infty(\Omega,\R)\cap C(\overline{\Omega},\R),
 \end{gather*}
 and we get that $U_n$ uniformly converges on $\overline{\Omega}$ to the unique
 continuous solution $u$ of
  \begin{gather*}
    \left\{
      \begin{array}{ll}
        P u=-f & \hbox{in $\mathcal{D}'(\Omega)$} \\
        u=\varphi & \hbox{on $\de\Omega$.}
      \end{array}
    \right.
 \end{gather*}
 From the WMP for $-U_n$ (recalling that $F_n\geq 0$ and $\varphi\geq0$), we derive
 $U_n\geq 0$ on $\overline{\Omega}$ ; this gives $u(x)=\lim_{n\to \infty} U_n(x)\geq 0$
 for all $x\in \overline{\Omega}$. This completes the proof.\qed




\begin{thebibliography}{SK}

\bibitem{AbbondanzaBonfiglioli}
 Abbondanza, B., Bonfiglioli, A.:
\emph{On the Dirichlet problem and the inverse mean value theorem
 for a class of divergence form operators},
  J. London Math. Soc., 1--26
  (2012); doi: 10.1112/jlms/jds050.


\bibitem{Amano}
 Amano, K.:
 \emph{A necessary condition for hypoellipticity of degenerate elliptic-parabolic operators},
 Tokyo J. Math. \textbf{2} (1979) 111--120.




\bibitem{BattagliaBonfiglioli}
 Battaglia, E., Bonfiglioli, A.:
 \emph{Normal families of functions for subelliptic operators
  and the theorems of Montel and Koebe}
  J. Math. Anal. Appl. \textbf{409} (2014), 1--12.




\bibitem{BellMohammed}
 Bell, D.R., Mohammed, S.-E. A.:
 \emph{An extension of H\"ormander's theorem for infinitely degenerate second-order operators},
 Duke Math. J. \textbf{78} (1995), 453--475.


 \bibitem{BonfLancJEMS}
  Bonfiglioli, A., Lanconelli, E.:
  \emph{Subharmonic functions in sub-Riemannian settings},
 J. Eur. Math. Soc. \textbf{15}  (2013), 387--441.
%
 \bibitem{BonfLancTommas}
  Bonfiglioli, A., Lanconelli, E., Tommasoli, A.:
 \emph{Convexity of average operators
 for subsolutions to subelliptic equations}
 Analysis \& PDE, \textbf{7} (2014), 345--373.


\bibitem{BLUlibro}
 Bonfiglioli, A., Lanconelli, E., Uguzzoni, F.:
 \emph{Stratified Lie Groups and Potential Theory for their sub-Laplacians},
  Springer Monographs in Mathematics, New York, NY, Springer 2007.


\bibitem{Bony}
 Bony, J.-M.:
 \emph{Principe du maximum, in{\'e}galit{\'e} de {H}arnack et unicit{\'e}
 du probl{\`e}me de {C}auchy pour les op{\'e}rateurs elliptiques
 d{\'e}g{\'e}n{\'e}r{\'e}s},
 Ann. Inst. Fourier (Grenoble),
 \textbf{19} (1969),
 277--304.
%

 \bibitem{Brelot}
 Brelot, M.:
 \emph{Axiomatique des fonctions harmoniques},
 Séminaire de Mathématiques Supérieures, \textbf{14} (Été, 1965),
 Les Presses de l'Université de Montréal, Montréal, 1969.



%



%
 \bibitem{Christ}
  Christ., M.:
  \emph{Hypoellipticity in the infinitely degenerate regime},
  in: Complex Analysis and Geometry (J.D. McNeal, ed.), Ohio State Univ.
  Math Res. Inst. Publ. \textbf{9}, Walter de Gruyter, Berlin, 2001,
  pp. 59--84.

%
\bibitem{ConstantinescuCornea}
 Constantinescu, C., Cornea, A.:
 \emph{On the axiomatic of harmonic functions. I},
 Ann. Inst. Fourier (Grenoble) \textbf{13} (1963), 373--388.

\bibitem{DeCiccoVivaldi}
 De Cicco, V.,  Vivaldi, M.A.:
 \emph{Harnack inequalities for Fuchsian type weighted elliptic equations},
 Comm. Partial Differential Equations \textbf{21} (1996), 1321--1347.


\bibitem{Dieudonne}
 Dieudonné, J.:
 \emph{Éléments d'analyse}. Tome VII. Chapitre XXIII. Première partie. Cahiers Scientifiques,
 Fasc. XL. Gauthier-Villars, Paris, 1978.

\bibitem{FabesJerisonKenig}
 Fabes, E.B., Jerison, D., Kenig, C.E.:
 \emph{The Wiener test for degenerate elliptic equations},
  Ann. Inst. Fourier (Grenoble) \textbf{32} (1982), 151--182.

\bibitem{FabesKenigJerison}
 Fabes, E.B., Kenig, C.E., Jerison, D.:
 \emph{Boundary behavior of solutions to degenerate elliptic equations},
 Conference on harmonic analysis in honor of Antoni Zygmund, Vol. I, II (Chicago, Ill., 1981), 577--589,
 Wadsworth Math. Ser., Wadsworth, Belmont, CA, 1983.

\bibitem{FabesKenigSerapioni}
 Fabes, E.B., Kenig, C.E., Serapioni, R.P.:
 \emph{The local regularity of solutions of degenerate elliptic equations},
 Comm. Partial Differential Equations \textbf{7} (1982), 77--116.

 \bibitem{Fedii}
 Fedi$\breve{\textrm{\i}}$, V.S.:
 \emph{On a criterion for hypoellipticity}, Math. USSR Sb. \textbf{14} (1971), 15--45.

\bibitem{FeffermanPhong2}
 Fefferman, C., Phong, D.H.:
 \emph{Subelliptic eigenvalue problems},
 Wadsworth Math. Ser. (Chicago, Ill., 1981), Wadsworth, Belmont, CA, 1983, pp. 590--606.

\bibitem{FeffermanPhong}
 Fefferman, C., Phong, D.H.:
 \emph{The uncertainty principle and sharp Gårding inequalities},
  Comm. Pure Appl. Math. \textbf{34} (1981), 285--331.



\bibitem{FollandStein}
 Folland, G.B., Stein, E.M.:
 \emph{Estimates for the $\bar \partial_b$ complex and analysis on the Heisenberg group},
 Commun. Pure Appl. Math. \textbf{27} (1974), 429--522.


\bibitem{GilbargTrudinger}
 Gilbarg, D., Trudinger, N.S.:
 \emph{Elliptic partial differential equations of second order}.
 Reprint of the 1998 edition. Classics in Mathematics, Springer-Verlag, Berlin, 2001.

\bibitem{Gutierrez}
 Gutiérrez, C.E.:
 \emph{Harnack's inequality for degenerate Schrödinger operators},
 Trans. Amer. Math. Soc. \textbf{312} (1989), 403--419.


%
\bibitem{Hormander}
 H{\"o}rmander, L.:
 \emph{Hypoelliptic second order differential equations},
  Acta Math. \textbf{119} (1967), 147--171.

   \bibitem{HormanderLibro}
 H{\"o}rmander, L.:
 \emph{The analysis of linear partial differential operators. I. Distribution theory and Fourier analysis}. Second edition. Springer Study Edition. Springer-Verlag, Berlin, 1990.
%
\bibitem{JerisonSanchez-Calle}
 Jerison, D.S., S\'anchez-Calle, A.:
 \emph{Subelliptic, second order differential operators},
 in: Complex Analysis III, Proc. Spec. Year, College Park 1985-86, Lect. Notes Math. \textbf{1277} (1987) 46--77.

\bibitem{Jurdjevic}
 Jurdjevic, V.:
 \emph{Geometric control theory},
 Cambridge Studies in Advanced Mathematics, \textbf{52}.
 Cambridge University Press, Cambridge, 1997.


\bibitem{Kohn65}
 Kohn, J.J.:
 \emph{Boundaries of complex manifolds},
 Proc. Conf. Complex Analysis, Minneapolis 1964, Springer-Verlag: New York, 81--94 (1965).


\bibitem{Kohn}
 Kohn, J.J.:
 \emph{Hypoellipticity of some degenerate subelliptic operators},
 J. Funct. Anal. \textbf{159} (1998), 203--216.

\bibitem{KohnNirenberg}
 Kohn, J.J., Nirenberg, L.:
 \emph{Non-coercive boundary value problems},
 Comm. Pure Appl. Math. \textbf{18} (1965), 443--492.




 \bibitem{KusuokaStroock}
 Kusuoka, S., Stroock, D.:
 \emph{Applications of the Malliavin calculus. II},
 J. Fac. Sci. Univ. Tokyo Sect. IA Math. \textbf{32} (1985), 1--76.

\bibitem{lanco_maxprinc}
 Lanconelli, E.:
 \emph{Maximum principles and symmetry results in sub-Riemannian settings}.
  In: Symmetry for elliptic PDEs, 17--33, Contemp. Math., \textbf{528},
  Amer. Math. Soc., Providence, RI, 2010.



\bibitem{LoebWalsh}
 Loeb, P., Walsh, B.:
 \emph{The equivalence of Harnack's principle and Harnack's inequality in the
 axiomatic system of Brelot},
 Ann. Inst. Fourier (Grenoble) \textbf{15} (1965), 597--600.

\bibitem{lopez-Gomez}
 López-Gómez, J.:
 \emph{The strong maximum principle},
 in: Mathematical analysis on the self-organization and self-similarity, 113--123,
 RIMS Kôkyûroku Bessatsu, B15, Res. Inst. Math. Sci. (RIMS), Kyoto, 2009.


%
 \bibitem{Montel}
  Montel, P.:
  \emph{Le\c{c}ons sur les familles normales des fonctions analytiques et leurs applications},
  Gauthier-Villars, Paris, 1927.


\bibitem{Morimoto}
 Morimoto, Y.:
 \emph{A criterion for hypoellipticity of second order differential
 operators}, Osaka J. Math.
 \textbf{24} (1987), 651--675.


\bibitem{OleinikRadkevic}
 Ole\u{\i}nik, O.A.,  Radkevi\v{c}, E.V.:
 \emph{Second Order Differential Equations with Nonnegative Characteristic Form},
 Amer. Math. Soc., RI/Plenum Press, New York (1973).


%
\bibitem{PucciSerrin}
 Pucci, P., Serrin, J.:
 \emph{The Maximum Principle},
 Progress in Nonlinear Differential Equations and their Applications,
 \textbf{73}, Birkhäuser Verlag, Basel, 2007.

\bibitem{RothschildStein}
 Rothschild, L.P., Stein, E.M.:
 \emph{Hypoelliptic differential operators and nilpotent groups},
 Acta Math. \textbf{137} (1977), 247--320.



\bibitem{Stein}
 Stein, E.M.: \emph{An example on the Heisenberg group related to the Lewy operator},
 Invent. Math. \textbf{69} (1982), 209--216.


%
\bibitem{Treves}
 {Treves, F.}:
 \emph{Topological vector spaces, distributions and kernels},
 Academic Press, London, 1967.

 \bibitem{Zamboni}
  Zamboni, P.:
  \emph{Hölder continuity for solutions of linear degenerate elliptic equations under minimal assumptions},
 J. Differential Equations \textbf{182} (2002), 121--140.
%
\end{thebibliography}
\end{document}